\documentclass[a4paper,10pt,titletoc]{amsart}

\usepackage{graphicx}      

\usepackage{bookmark}

\usepackage{epsfig}

\usepackage{amsmath}%

\graphicspath{{fig/}}

\usepackage{amsfonts}
\usepackage{amssymb}

\usepackage{tikz}
\usetikzlibrary{matrix,positioning,decorations.pathreplacing}

\usepackage{color}

\usepackage{enumitem}

\usepackage{ifthen}

\makeatletter
\DeclareOldFontCommand{\rm}{\normalfont\rmfamily}{\mathrm}
\DeclareOldFontCommand{\sf}{\normalfont\sffamily}{\mathsf}
\DeclareOldFontCommand{\tt}{\normalfont\ttfamily}{\mathtt}
\DeclareOldFontCommand{\bf}{\normalfont\bfseries}{\mathbf}
\DeclareOldFontCommand{\it}{\normalfont\itshape}{\mathit}
\DeclareOldFontCommand{\sl}{\normalfont\slshape}{\@nomath\sl}
\DeclareOldFontCommand{\sc}{\normalfont\scshape}{\@nomath\sc}
\makeatother

\usepackage{csquotes}  
\newcommand\q{\enquote}

\usepackage{physics}   

\newcommand{\ccat}[3]{{#1\, \underset{#3}{\lozenge}\,{#2}}}

\newcommand{\tm}{\times}
\newcommand\sat{\textup{sat}}

\DeclareMathOperator{\loc}{loc}

\DeclareMathOperator*{\esssup}{ess\,sup}

\newcommand \Liminf {\mathop{\underline{\lim}}}

\newcommand \eps {\varepsilon}

\newcommand \N   {\mathbb{N}}
\newcommand \R   {\mathbb{R}}

\newcommand \C   {\mathbb{C}}

\newcommand \K   {\mathcal{K}}
\newcommand \Kinf{\mathcal{K_\infty}}

\newcommand{\Uc}{\ensuremath{\mathcal{U}}}

\newcommand{\Dc}{\ensuremath{\mathcal{D}}}
\newcommand{\Sc}{\ensuremath{\mathcal{S}}}

\newcommand{\vertiii}[1]{{\left\vert\kern-0.25ex\left\vert\kern-0.25ex\left\vert #1 
    \right\vert\kern-0.25ex\right\vert\kern-0.25ex\right\vert}}

\newcommand{\sg}[1]{(#1(t))_{t\geq 0}} 

\newcommand \qrq   {\quad\Rightarrow\quad}

\newcommand \id  {\operatorname{id}}

\newcommand \re  {\mathrm{Re}}
\newcommand \im  {\mathrm{Im}}

\newcommand{\normt}[1]{{\left\vert\kern-0.25ex\left\vert\kern-0.25ex\left\vert #1 
		\right\vert\kern-0.25ex\right\vert\kern-0.25ex\right\vert}}

\renewcommand{\ker}{{\rm Ker}\,}

\newcommand{\dist}{{\rm dist}\,}

\newcommand{\Ah}{\hat{A}}

\newcommand{\Rh}{\hat{R}}

\newcommand{\clo}[1]{\overline{#1}}

\newcommand{\mir}[1]{{\color{red}\bf AM: #1}}     
\newcommand{\amc}[1]{{\color{black} #1}}           

\newif\ifMath					
\newif\ifEngi					

\newif\ifAndo              
													
\newif\ifExercises					
\newif\ifSolutions          
\newif\ifGerman							
\newif\ifEnglish						

\newif\ifnothabil						

\newif\ifFuture							

\newif\ifConf                    
\newif\ifJournal								 

\newif\ifNOTFORBOOK
\newif\ifFullVersion
\newif\ifExludedDueToSpaceReasons

\usepackage{xifthen}

\newcommand{\einsnorm}[2]{\ensuremath{
    \!\!\;\!\!\!\;
    \left\bracevert\!\!\!\!\!\left\bracevert
    \!
		\ifthenelse{\isempty{#2}}{#1}{#1(#2)}
    \!
      \right\bracevert\!\!\!\!\!\right\bracevert
    \!\!\;\!\!\!\;
  }}

\usepackage{xcolor}
\definecolor{blond}{rgb}{0.98, 0.94, 0.75}
	
\newlength\mytemplen
\newsavebox\mytempbox

\makeatletter
\newcommand\mybluebox{%
    \@ifnextchar[
       {\@mybluebox}%
       {\@mybluebox[0pt]}}

\def\@mybluebox[#1]{%
    \@ifnextchar[
       {\@@mybluebox[#1]}%
       {\@@mybluebox[#1][0pt]}}

\def\@@mybluebox[#1][#2]#3{
    \sbox\mytempbox{#3}%
    \mytemplen\ht\mytempbox
    \advance\mytemplen #1\relax
    \ht\mytempbox\mytemplen
    \mytemplen\dp\mytempbox
    \advance\mytemplen #2\relax
    \dp\mytempbox\mytemplen
    \colorbox{blond}{\hspace{1em}\usebox{\mytempbox}\hspace{1em}}}

\makeatother

\makeatletter
\let\origd=\d
\renewcommand*\d{
  \relax\ifmmode
    \mathrm{d}%
  \else
    \expandafter\origd
  \fi
}\makeatother

\usepackage{mathtools}

\makeatletter 
\newcommand{\pushright}[1]{\ifmeasuring@#1\else\omit\hfill$\displaystyle#1$\fi\ignorespaces}
\newcommand{\pushleft}[1]{\ifmeasuring@#1\else\omit$\displaystyle#1$\hfill\fi\ignorespaces}
\makeatother 

\newcounter{syscounter}
\newenvironment{sysnum}{\begin{list}{($\Sigma{\arabic{syscounter}}$)}%
{\settowidth{\labelwidth}{($\Sigma4$)}
\settowidth{\leftmargin}{($\Sigma4$)~}%
\usecounter{syscounter}}}
{\end{list}}

\newcounter{WPcounter}

\newtheorem{theorem}{Theorem}[section]
\newtheorem{lemma}[theorem]{Lemma}
\newtheorem{proposition}[theorem]{Proposition}
\newtheorem{corollary}[theorem]{Corollary}
\newtheorem{definition}[theorem]{Definition}
\newtheorem{example}[theorem]{Example}

\newtheorem{ass}{Assumption}[section]

\newtheorem*{solution*}{Solution}

\newtheorem{nnremark}[theorem]{\bf Remark}
\newenvironment{remark}{\begin{nnremark} \rm }{\hfill \hspace*{1pt}\hfill $\lrcorner$\end{nnremark}}

\usepackage{pgfplots}
\usepackage{pgf}
\usepackage{tikz} 
\usetikzlibrary{decorations.pathmorphing} 
\usepackage{tikz-3dplot}
\usepgfplotslibrary{fillbetween}
\usetikzlibrary{arrows}
\usetikzlibrary{graphs,decorations.pathmorphing,decorations.markings}
\usetikzlibrary{calc,math}
\usetikzlibrary{patterns}
\usetikzlibrary{shapes}
\usetikzlibrary{tikzmark}

\usetikzlibrary {positioning}
\usetikzlibrary{shadows}
\usetikzlibrary{patterns.meta}
\usetikzlibrary{plotmarks}

\usepackage{subcaption}  

\usepackage[absolute,overlay]{textpos}
\usepackage{scalefnt}   

\tikzdeclarepattern{
  name=hatch,
  parameters={\hatchsize,\hatchangle,\hatchlinewidth},
  bounding box={(-.1pt,-.1pt) and (\hatchsize+.1pt,\hatchsize+.1pt)},
  tile size={(\hatchsize,\hatchsize)},
  tile transformation={rotate=\hatchangle},
  defaults={
    hatch size/.store in=\hatchsize,hatch size=5pt,
    hatch angle/.store in=\hatchangle,hatch angle=0,
    hatch linewidth/.store in=\hatchlinewidth,hatch linewidth=.4pt,
  },
  code={
      \draw[line width=\hatchlinewidth] (0,0) -- (\hatchsize,\hatchsize);
  }
}

\pgfmathsetmacro\weight{1/2}
\pgfmathsetmacro\third{1/3}
\pgfmathsetmacro\twothirds{2/3}

\tikzset{degil/.style={
            decoration={markings,
            mark= at position 0.5 with {
                  \node[transform shape] (tempnode) {$/$};
                  }
              },
              postaction={decorate}
}
}

\usepgfplotslibrary{fillbetween}
\usetikzlibrary{intersections,through}

\makeatletter 
\tikzset{use path/.code=\tikz@addmode{\pgfsyssoftpath@setcurrentpath#1}}
\makeatother

\definecolor{manipulator-color}{RGB}{88,44,44}
\definecolor{manipulator-contour}{rgb}{0.0, 0.18, 0.39}  

\tikzset{>=latex} 

\Andofalse %

\nothabilfalse

\Exercisesfalse

\Solutionsfalse 

\Germanfalse

\Englishtrue

\Futurefalse

\begin{document}

\title[Well-posedness for semilinear evolution equations]{Well-posedness and properties of the flow for semilinear evolution equations}
%

\author{Andrii Mironchenko}
\address{Faculty of Computer Science and Mathematics, University of
  Passau, Germany}
\curraddr{}
\email{andrii.mironchenko@uni-passau.de}

\subjclass[2010]{34H05, 35K58, 35Q93, 37L15, 93A15, 93B52, 93C10, 93C25, 93D05, 93D09}

\keywords{Well-posedness, evolution equations, boundary control systems, infinite-dimensional systems, 
analytic systems}
\date{\today}
\begin{abstract}
We derive conditions for well-posedness of semilinear evolution equations with unbounded input operators. Based on this, we provide sufficient conditions for such properties of the flow map as Lipschitz continuity, bounded-implies-continuation property, boundedness of reachability sets, etc. These properties represent a basic toolbox for stability and robustness analysis of semilinear boundary control systems.

We cover systems governed by general $C_0$-semigroups, and analytic semigroups that may have both boundary and distributed disturbances. We illustrate our findings on an example of a Burgers' equation with nonlinear local dynamics and both distributed and boundary disturbances.
\end{abstract}

\maketitle

%
%
%
%
%
%
%
%
%
%
%
%
%


\section{Introduction}
\label{sec:Introduction}

\textbf{Semilinear evolution equations.}  In this work, we analyze the well-posedness and properties of the flow for semilinear evolution equations of the form
\begin{subequations}
\label{eq:SEE+admissible-intro} 
\begin{eqnarray}
\dot{x}(t) & = & Ax(t) + B_2f(x(t),u(t)) + Bu(t),\quad t>0,  \label{eq:SEE+admissible-intro1}\\
x(0)  &=&  x_0. \label{eq:SEE+admissible-intro2}
\end{eqnarray}
\end{subequations}
Here $A$ generates a strongly continuous semigroup over a Banach space $X$, the operators $B$ and $B_2$ are admissible with respect to some function space, and \amc{$f$ is Lipschitz continuous in the first variable} \amc{(see Assumption~\ref{Assumption1} for precise requirements on $f$)}.
This class of systems is rather general:
\begin{itemize}
	\item If $B$ and $B_2$ are bounded operators, \eqref{eq:SEE+admissible-intro} corresponds to the classic semilinear evolution equations covering broad classes of semilinear PDEs with distributed inputs. If $A$ is a bounded operator, such a theory was developed in \cite{DaK74}. In the case of unbounded generators $A$, we refer to \cite{Paz83}, \cite{Hen81}, \cite[Chapter 11]{CuZ20}, \cite{CaH98}, etc. 

	\item If $B_2=0$, and $B$ is an admissible operator, then \eqref{eq:SEE+admissible-intro} reduces to the class of general linear control systems, that fully covers linear boundary control systems (see \cite{CuZ20, JaZ12}, \cite{TuW09,TuW14}, \cite{EmT00} for an overview). In particular, this class includes linear evolution PDEs with boundary inputs.
	
	\item 
	\amc{Consider a linear system 
\begin{eqnarray}
\dot{x} = Ax + Bv,
\label{eq:linear-open-loop-intro}
\end{eqnarray}
with admissible $B$. Let us apply a feedback controller $v(x)=f(x,u_1)+u_2$ that is subject to additive actuator disturbance $u_2$ and further disturbance input $u_1$. Substituting this controller into \eqref{eq:linear-open-loop-intro}, we arrive 
at systems \eqref{eq:SEE+admissible-intro}, with $B_2=B$. 
}

	\item In \cite{HuP21}, it was shown that the class of systems \eqref{eq:SEE+admissible-intro} includes 2D Navier–Stokes equations (under certain boundary conditions) with in-domain inputs and disturbances. Furthermore, in \cite{HuP21} the authors have designed an error feedback controller that guarantees approximate local velocity output tracking for a class of reference outputs.
	Viscous Burgers' equation with nonlinear local terms and boundary inputs of Dirichlet or Neumann type falls into the class \eqref{eq:SEE+admissible-intro} as well.
		
	\item In \cite{Sch20}, it was shown that semilinear boundary control systems with linear boundary operators could be considered a special case of systems \eqref{eq:SEE+admissible-intro}. In this case, it suffices to consider $B_2$ as the identity operator.
Furthermore, in \cite{Sch20}, the well-posedness and input-to-state stability of a class of analytic boundary control systems with nonlinear dynamics and a linear boundary operator were analyzed with the methods of operator theory. 
\end{itemize}

\textbf{ISS for infinite-dimensional systems.} Our main motivation to analyze the systems \eqref{eq:SEE+admissible-intro} 
stems from the robust stability theory.
During the last decade, we have witnessed tremendous progress in robust stability analysis of nonlinear infinite-dimensional systems subject to unknown unstructured disturbances.
Input-to-state stability (ISS) framework admits a significant place in this development, striving to become a unifying paradigm for robust control and observation of PDEs and their interconnections, including ODE-PDE and PDE-PDE cascades \cite{KaK19, MiP20, Sch20}.

Powerful techniques proposed to analyze the ISS property include: criteria of ISS and ISS-like properties in terms of weaker stability concepts 
\cite{MiW18b}, \cite{JNP18, Sch19c}, constructions of ISS Lyapunov functions for PDEs with in-domain and/or boundary controls 
\cite{PrM12, TPT18, ZhZ18, EdM19}, efficient functional-analytic methods for the study of linear systems with unbounded input operators (including linear boundary control systems) 
\cite{ZhZ19b, JNP18, JLR08, JSZ19, KaK16b, LhS19, KaK19}, non-coercive ISS Lyapunov functions \cite{MiW18b, JMP20}, as well as small-gain stability analysis of finite \cite{DaM13, KaJ07, KaJ11, Mir21} and infinite networks, \cite{DaP20, KMS21, MKG21, KMZ23}, etc. 

To make this powerful machinery work for any given system, one needs to verify its well-posedness,  properties of reachability sets, and regularity of the flow induced by this system. Usually this is done for PDE systems in a case-by-case manner.
In this paper, motivated by \cite{Sch20}, we develop sufficient conditions that help to derive these crucial properties for systems \eqref{eq:SEE+admissible-intro}, which cover many important PDE systems.

%

\textbf{State of the art.} 
The systems \eqref{eq:SEE+admissible-intro} have been studied (up to the assumptions on $f$, and the choice of the space of admissible inputs) in \cite{NaB16} under the requirement that its linearization is an exponentially stable regular linear system in the sense of \cite{TuW09, TuW14, Sta05}.
\cite{NaB16} ensures local well-posedness of regular nonlinear systems assuming the Lipschitz continuity of nonlinearity, and invoking regularity of the linearization.
On this basis, the authors show in \cite{NaB16} that  
an error feedback controller designed for robust output regulation of a linearization of a regular nonlinear system
 achieves approximate local output regulation for \amc{the original regular nonlinear system.}

Control of systems \eqref{eq:SEE+admissible-intro} has been studied recently in several papers. In particular, in \cite{NZW19}, the exact controllability of a class of regular nonlinear systems was studied using back-and-forth iterations.
A problem of robust observability was studied for a related class of systems in \cite{JLZ15}.

Stabilization of linear port-Hamiltonian systems by means of nonlinear boundary controllers was studied in \cite{Aug16,RZG17}.
Bounded controls with saturations (a priori limitations of the input signal) have been employed for PDE control in \cite{PTS16,TMP18,MPW21}. 
Recently, several papers appeared that treat nonlinear boundary control systems within the input-to-state stability framework. 
Nonlinear boundary feedback was employed for the ISS stabilization of linear port-Hamiltonian systems in \cite{ScZ21}. 

Several types of infinite-dimensional systems, distinct from \eqref{eq:SEE+admissible-intro}, have been studied as well. 
One of such classes is time-variant infinite-dimensional semilinear systems that have been first studied (as far as the author is concerned) for systems without disturbances in \cite{JDP95}.
Recently, in \cite{Sch22}, sufficient conditions for well-posedness and uniform global stability have been obtained for scattering-passive semilinear systems (see \cite[Theorem 3.8]{Sch22}). 

Another important extension of \eqref{eq:SEE+admissible-intro} are semilinear systems with outputs. 
Such systems with globally Lipschitz nonlinearities have been analyzed in \cite[Section 7]{TuW14}, and it was shown that such systems are well-posed and forward complete provided that the Lipschitz constant is small enough. 
\amc{In \cite{HCZ19} employing a counterexample, it was shown that a linear transport equation with a locally Lipschitz boundary feedback might fail to be well-posed.} Well-posedness of incrementally scattering-passive nonlinear systems with outputs has been analyzed in \cite{SWT22} by applying Crandall-Pazy theorem \cite{CrP69} on generation of
nonlinear contraction semigroups to a Lax-Phillips nonlinear semigroup representing the system together with its inputs and outputs.

\textbf{Contribution.} 
Our first main result \amc{is Theorem~\ref{PicardCauchy} guaranteeing} (under proper conditions on $f$ and the input operators) the local existence and uniqueness of solutions for the system  \eqref{eq:SEE+admissible-intro}
\amc{with a locally essentially bounded input $u$. }

There are several existence and uniqueness theorems in the literature. For example, \cite[Proposition 3.2]{NaB16} covers semilinear systems with $L^\infty$-inputs; \cite[Theorem 7.6]{TuW14}, \cite[Lemma 2.8]{HJS22} treat the case of bilinear systems of various type, and \cite{Sch20} considers the case of systems with linearly bounded nonlinearities. 
In contrast to the usual formulations of such results (including a closely related result \cite[Proposition 3.2]{NaB16}), we also provide a uniform existence time for solutions that controls the maximal deviation of the trajectory from the given set of initial conditions. 

Next, we show \amc{in Theorems~\ref{thm:Global_well-posedness}, \ref{thm:SEE-as-control systems}} that under natural conditions, the system \eqref{eq:SEE+admissible-intro} is a well-posed control system in the sense of \cite{MiP20}. Finally, we study the fundamental properties of the flow map, such as Lipschitz continuity with respect to initial states, boundedness of reachability sets, boundedness-implies-continuation property, etc. 
These properties are important in their own right. \amc{Moreover, they are key components for the robust stability analysis} of systems 
\eqref{eq:SEE+admissible-intro} as we explained before.

%
%

The structure of semilinear evolution equations allows combining the \q{linear} methods of admissibility theory with \q{nonlinear} methods, such as fixed point theorems and Lyapunov methods. We consider the case of general $C_0$-semigroups and the special case of analytic semigroups, for which one can achieve stronger results. 
This synergy of tools is one of the novelties of this paper. For systems without inputs and without the presence of unbounded operators, the existence and uniqueness results as well as the properties of the flow are classical both for general and analytic case \cite{Paz83,Hen81}.
To show the applicability of our methods, we analyze well-posedness of semilinear parabolic systems with Dirichlet boundary inputs (motivated by \cite[p. 57]{Hen81}). Also, we  reformulate semilinear boundary control systems in terms of evolution equations, which makes our results applicable to boundary control systems as well. 

As argued at the previous pages, having developed conditions ensuring the well-posedness and \q{nice} properties of the flow map of systems \eqref{eq:SEE+admissible-intro}, we can analyze the ISS of \eqref{eq:SEE+admissible-intro} via such powerful tools as coercive and non-coercive ISS Lyapunov functions \cite{JMP20}, ISS superposition theorems \cite{MiW18b}, small-gain theorems for general systems \cite{MKG21}, etc. We expect that this will \amc{help to prove many results available for particular PDE systems, in a more general fashion. E.g., see \cite{Sch20} for an abstract version of the results obtained for particular classes of parabolic systems in \cite{ZhZ18}.}  
To make the paper accessible for the researchers trained primarily in nonlinear control and nonlinear ISS theory, we spell out the proofs in great detail with a tutorial flavor.





\textbf{Notation.} By $\N$, $\R$, $\R_+$, we denote the sets of natural, real, and nonnegative real numbers, respectively. 
$\clo{S}$ denotes the closure of a set $S$ (in a given topology). 

\amc{By $t \to a \pm 0$, we denote the fact that $t$ approaches $a$ from the right/left.}

Vector spaces considered in this paper are assumed to be real.

Let $S$ be a normed vector space.  
The distance from $z\in S$ to the set $Z \subset S$ we denote by 
\amc{$\dist(z,Z):=\inf\{\|z-y\|_S: y \in Z \}$.} 
We denote an open ball of radius $r$ around $Z \subset S$ by 
$B_{r,S}(Z):=\{y \in X:\dist(y,Z)<r\}$, and we set also $B_{r,S}(x):=B_{r,S}(\{x\})$ for $x \in X$, and $B_{r,S}:=B_{r,S}(0)$. 
If $S =X$ (the state space of the system), we write for short $B_{r}(Z):=B_{r,X}(Z)$, $B_{r}(x):=B_{r,X}(x)$, etc.

Denote by $\K$ the class of continuous strictly increasing functions $\gamma:\R_+\to\R_+$, satisfying $\gamma(0)=0$.
$\Kinf$ denotes the set of unbounded functions from $\K$.


For normed vector spaces $X,U$, denote by $L(X,U)$ the space of bounded linear operators from $X$ to $U$.
We endow $L(X,U)$ with the standard operator norm 
$\|A\|:=\sup_{\|x\|_X=1}\|Ax\|_U$.
We write for short $L(X):=L(X,X)$. 
\amc{By $C(X,U)$ we denote the space of continuous maps from $X$ to $U$. Similarly, by $C(\R_+,X)$ we understand the space of continuous maps from $\R_+$ to $X$.} 
\amc{The domain of definition, kernel, and image of an operator $A$ we denote by $D(A)$, $\ker(A)$, and $\im(A)$ respectively.} 
By $\sigma(A)$, we denote the spectrum of a closed operator $A:D(A)\subset X \to X$, and by $\rho(A)$ the resolvent set of $A$.  
We denote by $\omega_0(T)$ the growth bound of a $C_0$-semigroup $T$. 

Let $X$ be a Banach space, and let $I$ be a closed subset of $\R$. We define for $p\in[1,\infty)$ the following spaces of vector-valued functions
\begin{eqnarray*}
M(I,X) &:=& \{f: I \to X: f \text{ is strongly measurable} \},\\
L^p(I,X) &:=& \Big\{f \in M(I,X): \|f\|_{L^p(I,X)}:=\Big(\int_I\|f(s)\|^p_Xds\Big)^{\frac{1}{p}} < \infty \Big\},\\
L^p_{\loc}(\R_+,X) &:=& \{\amc{f|_{[0,t]}} \in L^p([0,t],X)\quad\forall t>0\},\\
L^\infty(I,X) &:=& \{f \in M(I,X): \|f\|_{L^\infty(I,X)}:=\esssup_{s\in I}\|f(s)\|_X < \infty \},\\
L^\infty_{\loc}(I,X) &:=& \{f \in L^\infty([0,t],X)\quad\forall t>0\}.
\end{eqnarray*}

Denote also $L^p(a,b)$ $:=L^p([a,b],\R)$, where $p\in[1,\infty]$. The space $H^{k}(a,b)$, $k\in\N$, is a Sobolev space of functions $u \in L^2(a,b)$, such that for each natural $j\leq k$, the weak derivative $u^{(j)}$ exists and belongs to $L^2(a,b)$. $H^{k}(a,b)$ is endowed with the norm 
$u\mapsto \Big(\sum_{j \leq k}\int_a^b{\big| u^{(j)} (x)\big|^2 dx} \Big)^{\frac{1}{2}}$. 
	$H^{k}_0(a,b)$ denotes the closure of smooth functions with compact support in $(a,b)$ in the norm of $H^{k}(a,b)$,\quad $k\in\N$.

\section{General class of systems}

We start with a general definition of a control system \amc{that we adopt from \cite{MiP20}.}
\index{control system}
\begin{definition}
\label{Steurungssystem}
Consider the triple $\Sigma=(X,\Uc,\phi)$ consisting of 
\index{state space}
\index{space of input values}
\index{input space}
\begin{enumerate}[label=(\roman*)]  
    \item A normed vector space $(X,\|\cdot\|_X)$, called the \emph{state space}, endowed with the norm $\|\cdot\|_X$.
    \item A normed vector \emph{space of inputs} $\Uc \subset \{u:\R_+ \to U\}$          
endowed with a norm $\|\cdot\|_{\Uc}$, where $U$ is a normed vector \emph{space of input values}.
We assume that the following two axioms hold:
                    
\emph{The axiom of shift invariance}: for all $u \in \Uc$ and all $\tau\geq0$ the time
shift $u(\cdot + \tau)$ belongs to $\Uc$ with \mbox{$\|u\|_\Uc \geq \|u(\cdot + \tau)\|_\Uc$}.

\emph{The axiom of concatenation}: for all $u_1,u_2 \in \Uc$ and for all $t>0$ the \emph{concatenation of $u_1$ and $u_2$ at time $t$}, defined by
\begin{equation}
\amc{\big(\ccat{u_1}{u_2}{t}\big)(\tau):=}
\begin{cases}
u_1(\tau), & \text{ if } \tau \in [0,t], \\ 
u_2(\tau-t),  & \text{ otherwise},
\end{cases}
\label{eq:Composed_Input}
\end{equation}
belongs to $\Uc$.

    \item A map $\phi:D_{\phi} \to X$, $D_{\phi}\subseteq \R_+ \times X \times \Uc$ (called \emph{transition map}), such that for all $(x,u)\in X \tm \Uc$ it holds that $D_{\phi} \cap \big(\R_+ \times \{(x,u)\}\big) = [0,t_m)\tm \{(x,u)\} \subset D_{\phi}$, for a certain $t_m=t_m(x,u)\in (0,+\infty]$.
		
		The corresponding interval $[0,t_m)$ is called the \emph{maximal domain of definition} of $t\mapsto \phi(t,x,u)$.
		
\end{enumerate}
The triple $\Sigma$ is called a \emph{(control) system}, if the following properties hold:
\index{property!identity}

\begin{sysnum}
    \item\label{axiom:Identity} \emph{The identity property:} for every $(x,u) \in X \times \Uc$
          it holds that $\phi(0, x,u)=x$.
\index{causality}
    \item \emph{Causality:} for every $(t,x,u) \in D_\phi$, for every $\tilde{u} \in \Uc$, such that $u(s) =
          \tilde{u}(s)$ for all $s \in [0,t]$ it holds that $[0,t]\tm \{(x,\tilde{u})\} \subset D_\phi$ and $\phi(t,x,u) = \phi(t,x,\tilde{u})$.
    \item \label{axiom:Continuity} \emph{Continuity:} for each $(x,u) \in X \times \Uc$ the map $t \mapsto \phi(t,x,u)$ is continuous on its maximal domain of definition.
\index{property!cocycle}
        \item \label{axiom:Cocycle} \emph{The cocycle property:} for all
                  $x \in X$, $u \in \Uc$, for all $t,h \geq 0$ so that $[0,t+h]\tm \{(x,u)\} \subset D_{\phi}$, we have
\[
\phi\big(h,\phi(t,x,u),u(t+\cdot)\big)=\phi(t+h,x,u).
\]
\end{sysnum}

\end{definition}

Definition~\ref{Steurungssystem} can be viewed as a direct generalization, and a unification of the concepts of strongly continuous nonlinear semigroups \cite{CrP69, CrL71} with abstract linear control systems \cite{Wei89b}.

This class of systems encompasses control systems generated by ordinary
differential equations (ODEs), switched systems, time-delay systems,
evolution partial differential equations (PDEs), abstract differential
equations in Banach spaces and many others \cite[Chapter 1]{KaJ11b}.

\begin{definition}
\label{def:FC_Property} 
\index{forward completeness}
We say that a \amc{control system $\Sigma=(X,\Uc,\phi)$} is \emph{forward complete (FC)}, if 
$D_\phi = \R_+ \tm X\tm\Uc$, that is for every $(x,u) \in X \times \Uc$ and for all $t \geq 0$ the value $\phi(t,x,u) \in X$ is well-defined.
\end{definition}

Forward completeness alone does not imply, in general, the existence of any uniform bounds on the trajectories emanating from bounded balls that are subject to uniformly bounded inputs \cite[Example 2, p. 1612]{MiW18b}. Systems exhibiting such bounds deserve a special name.
\begin{definition}
\label{def:BRS}
\index{bounded reachability sets}
\index{BRS}
We say that \emph{\amc{a control system $\Sigma=(X,\Uc,\phi)$}  has bounded reachability sets (BRS)}, if for any $C>0$ and any $\tau>0$ it holds that 
\[
\sup\big\{
\|\phi(t,x,u)\|_X : \|x\|_X\leq C,\ \|u\|_{\Uc} \leq C,\ t \in [0,\tau]\big\} < \infty.
\]
\end{definition}

For a wide class of control systems, the boundedness of a solution implies the possibility of prolonging it to a larger interval, see \cite[Chapter 1]{KaJ11b}. Next, we formulate this property for abstract systems:
\begin{definition}
\label{def:BIC} 
\index{property!boundedness-implies-continuation}
\index{BIC}
We say that a \amc{control system $\Sigma=(X,\Uc,\phi)$}  satisfies the \linebreak \emph{boundedness-implies-continuation (BIC) property} if for each
$(x,u)\in X \tm \Uc$ with $t_m(x,u)<\infty$ it holds that 
\[
\limsup_{t \to t_m(x,u)-0}\|\phi(t,x,u)\|_X = \infty.
\]
\end{definition}

\section{Semilinear evolution equations with unbounded input operators}
\label{sec:Semilinear boundary control systems}

Consider \amc{a Cauchy problem for} infinite-dimensional evolution equations of the form
\begin{subequations}
\label{eq:SEE+admissible} 
\begin{eqnarray}
\dot{x}(t) & = & Ax(t) + B_2f(x(t),u(t)) + Bu(t),\quad t>0,  \label{eq:SEE+admissible-1}\\
x(0)  &=&  x_0, \label{eq:SEE+admissible-2}
\end{eqnarray}
\end{subequations}
where $A: D(A)\subset X \to X$ generates a strongly continuous semigroup $T=\sg{T}$ of boun\-ded linear operators on a Banach space $X$; $U$ is a \amc{Banach space} of input values, and $x_0\in X$ is a given initial condition. 
As the input space, we take $\Uc:=L^\infty(\R_+,U)$.

The map $f: X\tm U \to V$ is defined on the whole $X \tm U$ and maps to a Banach space $V$. 
Furthermore, $B\in L(U,X_{-1})$ and $B_2 \in L(V,X_{-1})$. Here the extrapolation space $X_{-1}$ is the closure of $X$ in the norm $x \mapsto \|(aI-A)^{-1}x\|_X$, $x \in X$, where $a \in \rho(A)$ (different choices of $a \in\rho(A)$ induce equivalent norms on $X$). 
\amc{Note that the operators $B$ and $B_2$ are unbounded, if they are understood as operators that map to $X$.}

%

\subsection{Admissible input operators and mild solutions}
\label{sec:Admissible input operators and mild solutions}

First, consider the linear counterpart of the system \eqref{eq:SEE+admissible}.
\begin{subequations}
\label{eq:LEE+admissible} 
\begin{eqnarray}
\dot{x}(t) & = & Ax(t) + Bu(t),\quad t>0,  \label{eq:LEE+admissible-1}\\
x(0)  &=&  x_0, \label{eq:LEE+admissible-2}
\end{eqnarray}
\end{subequations}
for the same $A, B$ as above. As the image of $B$ does not necessarily lie in $X$, one has to be careful when defining the concept of a solution for \eqref{eq:LEE+admissible}. Since $B \in L(U,X_{-1})$, it is natural to consider the system 
\eqref{eq:LEE+admissible} on the space $X_{-1}$.
Note that the semigroup  $(T(t))_{t\ge 0}$ extends uniquely to a strongly continuous semigroup  $(T_{-1}(t))_{t\ge 0}$ on $X_{-1}$ whose generator $A_{-1}$ acting in $X_{-1}$ is an extension of $A$ with $D(A_{-1}) = X$, see, e.g.,\ \cite[Section II.5]{EnN00}.
Recall the definitions of the spaces $L^p$, $L^p_{\loc}$ from Section~\ref{sec:Introduction}.


The mild solution of \eqref{eq:LEE+admissible} for any $x \in X$ and $u \in L^1_{\loc}(\R_+,U)$ is given by 
\[
\phi_L(t,x,u) = T(t)x + \int_0^t T_{-1}(t-s)Bu(s)ds,\quad t\geq 0.
\]
The integral term here, however, belongs in general to $X_{-1}$.

Thus, the existence and uniqueness of a mild solution depend on whether\linebreak $\int_0^t T_{-1}(t-s)Bu(s)ds \in X$. 
This leads to the following concept:
\begin{definition}
\label{def:q-admissibility}
\index{operator!$q$-admissible}
Let $q\in[1,\infty]$. The operator $B\in L(U,X_{-1})$ is called a \emph{$q$-admissible control operator} for $(T(t))_{t\ge 0}$, if
there is $t>0$ so that
\begin{eqnarray}
u\in L^q_{\loc}(\R_+,U) \qrq \int_0^t T_{-1}(t-s)Bu(s)ds\in X.
\label{eq:q-admissibility}
\end{eqnarray}
\end{definition}
Define for each $t\geq 0$ an operator $\Phi(t): L^q_{\loc}(\R_+,U) \to X_{-1}$ by
\[
\Phi(t)u := \int_0^t T_{-1}(t-s)Bu(s)ds.
\]
Note that as $B \in L(U,X_{-1})$, the operators $\Phi(t)$ are well-defined as maps from $L^1_{\loc}(\R_+,U)$ to $X_{-1}$ for all $t$. The next result 
(see \cite[Proposition 4.2]{Wei89b}, \cite[Proposition 4.2.2]{TuW09}) shows that $q$-admissibility of $B$ ensures that the image of $\Phi(t)$ is in $X$ for all $t \geq 0$ and $\Phi(t)\in L(L^q(\R_+,U),X)$ for all $t>0$.

\begin{proposition}
\label{prop:Restatement-admissibility}
Let $X, U$ be Banach spaces and let $q \in [1,\infty]$ be given. Then $B \in L(U,X_{-1})$ is $q$-admissible if and only if
for all $t>0$ there is $h_t>0$ so that for all $u \in L^{q}_{\loc}(\R_+,U)$ it holds that $\Phi(t)u \in X$ and
\begin{equation}
\label{eq:admissible-operator-norm-estimate}
\left\| \int_0^t T_{-1}(t-s)Bu(s)\,ds\right\|_X \le h_t \|u\|_{L^q([0,t],U)}.
\end{equation}
The function $t\mapsto h_t$ we assume wlog to be nondecreasing in $t$. 
\end{proposition}

An important consequence of Proposition~\ref{prop:Restatement-admissibility} is that well-posedness (and thus forward completeness) of the system \eqref{eq:LEE+admissible} already implies the boundedness of reachability sets property for \eqref{eq:LEE+admissible}, with a bound given by \eqref{eq:admissible-operator-norm-estimate}.

As $t\mapsto h_t$ is nondecreasing in $t$, there is a limit $h_0:=\lim_{t\to +0}h_t \geq 0$, which is not necessarily zero. 
Operators for which $h_0=0$ deserve a special name.
\begin{definition}
\label{def:zero-class-admissibility} 
Let $q \in [1,\infty]$. A $q$-admissible operator $B\in L(U, X_{-1})$ is called \emph{zero-class $q$-admissible}, if the constants $(h_t)_{t>0}$ can be chosen such that $h_0=0$.
\end{definition}
All $B \in L(U,X)$ are zero-class 1-admissible. If $X$ is reflexive, then $1$-admissible operators are necessarily bounded.
At the same time, there are unbounded zero-class admissible operators, see Proposition~\ref{prop:ISS-analytic-systems}. Consider \cite[Examples 3.8, 3.9]{JPP09} for unbounded admissible observation operators that are not zero-class admissible.

The above considerations motivate us to impose
\begin{ass}
\label{ass:Admissibility} 
The operator $B\in L(U,X_{-1})$ is $\infty$-admissible, and the map $(t,u) \mapsto \Phi(t)u$ is continuous on $\R_+\tm L^\infty(\R_+,U)$.

In particular, this assumption holds if $B$ is a $q$-admissible operator with $q<\infty$, see \cite[Proposition 2.3]{Wei89b}.
\end{ass}

To define the concept of a mild solution, we also require the following:
\begin{ass}
\label{ass:Integrability} 
We assume that $B_2$ is zero-class $\infty$-admissible and for all $u\in L^\infty(\R_+,U)$ and any $x\in C(\R_+,X)$ the map $s\mapsto f\big(x(s),u(s)\big)$ is in $L^\infty_{\loc}(\R_+,V)$.

Due to \cite[Proposition 2.5]{JNP18}, these conditions ensure that for above $x,u$ the map 
\begin{eqnarray}
\label{eq:Integrability-map} 
t \mapsto \int_0^t T_{-1}(t-s)B_2f\big(x(s),u(s)\big)ds
\end{eqnarray}
is well-defined and continuous on $\R_+$.
\end{ass}

\begin{remark}
\label{rem:Validity-of-integrability-assumption} 
Assumption~\ref{ass:Integrability} holds, in particular, if $B_2 \in L(V,X)$, and
\begin{enumerate}[label=(\roman*)]
	\item $f(x,u)=g(x) + Ru$, $x \in X$, $u\in U$, where $R\in L(U,V)$, and $g$ is continuous on $X$.
Indeed, for a continuous $x$, the map	$s\mapsto g\big(x(s)\big)$ is continuous either, and thus Riemann integrable. 
The map $s\mapsto T(t-s) B_2Ru(s)$ is Bochner integrable for any $u \in L^1_{\loc}(\R_+,U)$ by \cite[Proposition 1.3.4]{ABH11}, \cite[Lemma 10.1.6]{JaZ12}. 
This ensures that Assumption~\ref{ass:Integrability} holds.

\item If $f$ is continuous on $X \tm U$, and $u$ is piecewise right-continuous, then the map 
 $s\mapsto f\big(x(s),u(s)\big)$ is also piecewise right-continuous, and thus it is Riemann integrable.

\item (ODE systems). Let $X=\R^n$, $U=\R^m$, $A=0$ (and thus $T(t)=\id$ for all $t$), $B_2=\id$, $B=0$, and $f$ be continuous on $X \tm U$. With these assumptions the equations \eqref{eq:SEE+admissible} take the form
\begin{align}
\label{eq:ODE}
\dot{x} = f(x,u).
\end{align}
Then for each $u\in L^\infty(\R_+,U)$ and each $x \in C(\R_+,X)$ the map $s\mapsto f(x(s),u(s))$ is Lebesgue integrable, and thus Assumption~\ref{ass:Integrability} holds.

Indeed, as $x$ is a solution of \eqref{eq:ODE} on $[0,\tau)$, $x$ is continuous on $[0,\tau)$. 
By assumptions, $u$ is measurable on $[0,\tau)$, and $f$ is continuous on $\R^n \tm \R^m$. Arguing similarly to  
\cite[Proposition 7]{RoF10} (where it was shown that a composition of a continuous and measurable function defined on a measurable set $E$ is measurable on $E$), we see that the map $q : [0,\tau)\to\R^n$, $q(s):= f(x(s),u(s))$, is a measurable map.
As $u$ is essentially bounded, and $x$ and $f$ map bounded sets into bounded sets, $q$ is essentially bounded on $[0,\tau)$. 
Thus, $q \in L^{\infty}(\R_+,\R^n)$, and thus $q$ is integrable on $[0,\tau)$.
\end{enumerate}
\end{remark}

Next we define mild solutions of \eqref{eq:SEE+admissible}.
\begin{definition}[Mild solutions]
\label{def:Mild-solution}
\index{solution!mild}
Let Assumptions~\ref{ass:Admissibility}, \ref{ass:Integrability} hold and $\tau>0$ be given. 
A function $x \in C([0,\tau], X)$ is called a \emph{mild solution of \eqref{eq:SEE+admissible} on $[0,\tau]$} corresponding to certain $x_0\in X$ and $u \in L^\infty_{\loc}(\R_+,U)$, if $x$ solves the integral equation
\begin{align}
\label{eq:SEE+admissible_Integral_Form}
x(t)=T(t) x_0 + \int_0^t T_{-1}(t-s) B_2f\big(x(s),u(s)\big)ds + \int_0^t T_{-1}(t-s) Bu(s)ds. 
\end{align}
Here the integrals are Bochner integrals of $X_{-1}$-valued maps.

We say that $x:\R_+\to X$ is a \emph{mild solution of \eqref{eq:SEE+admissible} on $\R_+$} corresponding to 
certain $x_0\in X$ and $u \in L^\infty_{\loc}(\R_+,U)$, if \amc{$x|_{[0,\tau]}$} is a mild solution of \eqref{eq:SEE+admissible} (with $x_0, u$) on $[0,\tau]$ for all $\tau>0$.
\end{definition}

\subsection{Local existence and uniqueness}
\label{sec:Local existence and uniqueness}

Assumptions~\ref{ass:Admissibility}, \ref{ass:Integrability} guarantee that the integral terms in \eqref{eq:SEE+admissible_Integral_Form}
are well-defined. To ensure the existence and uniqueness of mild solutions, we impose further restrictions on $f$.

Recall the notation $B_{C,U} = \{v \in U:\|v\|_U < C\}$ and $B_{C} = \{v \in X:\|v\|_X < C\}$.
\begin{definition}
We call $f:X \times U \to V$
\begin{enumerate}[label=(\roman*)]
\index{function!Lipschitz continuous}
	\item \emph{Lipschitz continuous  (with respect to the first argument) on bounded subsets of $X$} if for any
$C>0$ there is $L(C)>0$, such that $\forall x,y \in B_C$, $\forall v \in B_{C,U}$ it holds that
\begin{eqnarray}
\|f(y,v)-f(x,v)\|_V \leq L(C) \|y-x\|_X.
\label{eq:Lipschitz}
\end{eqnarray}	
	
	\item \emph{uniformly globally Lipschitz continuous (with respect to the first argument)} if \eqref{eq:Lipschitz} holds for all $x,y \in X$, and all $v \in U$ with \amc{a constant $L$} that does not depend on $x,y,v$.
\end{enumerate}
\end{definition}


We omit the indication \q{with respect to the first argument} wherever this is clear
from the context. 

For the well-posedness analysis, we rely on the following assumption on the nonlinearity $f$ in \eqref{eq:SEE+admissible}.

\begin{ass}
\label{Assumption1} 
The nonlinearity $f$ satisfies the following properties:
\begin{itemize}
    \item[(i)] $f: X \times U \to V$ is Lipschitz continuous on bounded subsets of $X$.
    \item[(ii)] $f(x,\cdot)$ is continuous for all $x \in X$.
		\item[(iii)] There exist $\sigma \in \Kinf$ and $c>0$ so that for all $u \in U$ the following holds:
\begin{eqnarray}
\|f(0,u)\|_V \leq \sigma(\|u\|_U) + c.
\label{eq:f0u_estimate}
\end{eqnarray}
\end{itemize}
\end{ass}

Recall the notation for the distances and balls in normed vector spaces, introduced in the end of Section~\ref{sec:Introduction}. 
Finally, for a set $\Sc \subset U$, denote the set of inputs with essential \amc{image} in $\Sc$ as $\Uc_\Sc$:
\begin{align}
\label{eq:Inputs-constrained}
\Uc_\Sc:=\{u\in\Uc: u(t) \in \Sc, \text{ for a.e. } t\in\R_+\}.
\end{align}

We start with the following sufficient condition for the existence and uniqueness of solutions of a system \eqref{eq:SEE+admissible} with inputs in $L^\infty(\R_+,U)$.

\amc{Recall the notation $h_0:=\lim_{t\to +0}h_t$, where $h_t$ is defined as in \eqref{eq:admissible-operator-norm-estimate}.}
\index{theorem!Picard-Lindel\"of}
\index{Picard-Lindel\"of theorem}
\begin{theorem}[Picard-Lindel\"of theorem]
\label{PicardCauchy}
Let Assumptions~\ref{ass:Admissibility}, \ref{ass:Integrability}, \ref{Assumption1} hold. 

Assume that $(T(t))_{t\ge 0}$ satisfies for certain $M \geq 1$, $\lambda>0$ the estimate
\begin{eqnarray}
\|T(t)\| \leq Me^{\lambda t},\quad t\geq 0.
\label{eq:Picard-Lindeloef-bounds-semigroup}
\end{eqnarray} 
For any compact set $Q \subset X$, any $r>0$, any bounded set $\Sc \subset U$, and any $\delta>0$, there is a time $t_1 = t_1(Q,r,\Sc,\delta)>0$, such that for any $w\in Q$, for any $x_0 \in W:=B_r(w)$, and for any $u\in\Uc_{\Sc}$ there is a unique mild solution of \eqref{eq:SEE+admissible} on $[0,t_1]$, and  $\phi([0,t_1],x_0,u) \subset B_{Mr+h_0\|u\|_{L^\infty([0,t_1],U)}+\delta}(w)$.
\end{theorem}

\begin{proof}
First, we show the claim for the case if $Q$ is a single point in $X$, that is, $Q=\{w\}$, for some $\omega\in X$. 
Pick any $C>0$ such that $W := B_r(w) \subset B_C$, and $\Uc_\Sc \subset B_{C,\Uc}$. Pick any $u \in \Uc_\Sc$. Also take any $\delta>0$, and consider the following sets (depending on the parameter $t>0$):
\begin{eqnarray}
\hspace{10mm}Y_{t}:= \big\{x \in C([0,t],X): \sup_{s \in[0,t]} \|x(s) - w\|_X\leq Mr + h_0\|u\|_{L^\infty([0,t],U)} + \delta \big\},
\label{eq:Y_T_Def}
\end{eqnarray}
endowed with the metric 
$\rho_{t}(x,y):=\sup_{s \in [0,t]} \|x(s)-y(s)\|_X$. 
As the sets $Y_t$ are closed subsets of the Banach spaces $C([0,t],X)$, for all $t>0$, the space $Y_t$ is a complete metric space.

Pick any $x_0 \in W$. We are going to prove that for small enough $t$, the spaces $Y_{t}$  are invariant under the operator
$\Phi_u$, defined for any $x \in Y_{t}$ and all $\tau \in [0,t]$ by
\begin{align}
\label{eq:Phi-map-PL-Theorem}
\Phi_u(x)(\tau) 
&\amc{:=} T(\tau)x_0 + \int_0^\tau T_{-1}(\tau-s)B_2f\big(x(s),u(s)\big)ds  + \int_0^\tau T_{-1}(\tau-s) Bu(s)ds.
\end{align}
By Assumptions~\ref{ass:Admissibility}, \ref{ass:Integrability}, the function $\Phi_u(x)$ is continuous for any $x \in Y_t$.
\ifnothabil	\sidenote{\mir{and by Theorem~\ref{Milde_Loesungen}}}\fi

Fix any $t>0$ and pick any $x \in Y_{t}$.
As $x_0 \in W=B_r(w)$, there is $a \in B_r$ such that $x_0=w+a$.
 
Then for any $\tau<t$, it holds that
\begin{align*}
\|\Phi_{u}(x)&(\tau)-w\|_X \\
&\leq   \Big\|T(\tau)x_0 - w\Big\|_X + \Big\|\int_0^\tau T_{-1}(\tau-s) Bu(s)ds\Big\|_X \\
&\qquad\qquad\qquad\qquad\qquad + \Big\|\int_0^\tau T_{-1}(\tau-s) B_2f(x(s),u(s))ds\Big\|_X \\
&\leq  \|T(\tau)(w+a) - w \|_X +  h_\tau \|u\|_{L^\infty([0,\tau],U)} \\
&\qquad\qquad\qquad\qquad\qquad + c_\tau \|f(x(\cdot),u(\cdot))\|_{L^\infty([0,\tau],V)} \\
			&\leq \|T(\tau)w - w\|_X + \|T(\tau)a\|_X + h_\tau \|u\|_{L^\infty([0,\tau],U)}\\
			&\quad+ c_\tau \|f(x(\cdot),u(\cdot)) - f(0,u(\cdot))\|_{L^\infty([0,\tau],V)} + c_\tau \|f(0,u(\cdot))\|_{L^\infty([0,\tau],V)}.
\end{align*}
Now for all $s\in[0,t]$
\begin{align*}
\|x(s)\|_X &\leq \|w\|_X+Mr+ h_0\|u\|_{L^\infty([0,t],U)} +\delta\\
&\leq M(\|w\|_X+r) + h_0 C+\delta\leq (M+h_0)C+\delta =:K.
\end{align*}
In view of Assumption~\ref{Assumption1}(iii), it holds that 
\[
\|f(0,u(s))\|_V \leq \sigma(\|u(s)\|_U) + c,\quad \text{ for a.e. } s \in [0,t].
\]
As $M\geq 1$, it holds that $K>C$, and the Lipschitz continuity of $f$ on bounded balls ensures that there is $L(K)>0$, such that for all $\tau\in[0,t]$ 			
\begin{align*}
\|\Phi_{u}(x)(\tau)-w\|_X &\leq \|T(\tau)w - w\|_X + Me^{\lambda t} r + h_t \|u\|_{L^\infty([0,t],U)} \\
					&\qquad\qquad+ c_\tau \big( L(K) \|x\|_{L^\infty([0,t],X)} + \sigma(\|u\|_{L^\infty([0,t],U)}) + c\big)\\
			&\leq \|T(\tau)w - w\|_X + Me^{\lambda t} r + h_t \|u\|_{L^\infty([0,t],U)} \\
					&\qquad\qquad+ c_t \big( L(K)K + \sigma(C) + c\big).
\end{align*}
Since $T$ is a strongly continuous semigroup, as $h_t\to h_0$ whenever $t\to +0$, and 
\amc{since $c_t\to 0$ as $t\to +0$, there exists $t_1$,} such that 
\[
\|\Phi_u(x)(t)-w\|_X \leq Mr+ h_0 \|u\|_{L^\infty([0,t_1],U)} +\delta,\quad \text{ for all } t \in [0,t_1].
\]
This means, that $Y_{t}$ is invariant with respect to $\Phi_u$ for all $t \in (0,t_1]$, and $t_1$ does not depend on the choice of $x_0 \in W$.

Now pick any $t>0$, $\tau \in [0, t]$, and any $x, y \in Y_{t}$. \amc{It holds} that
\begin{align*}
\|\Phi_u(x)(\tau) - \Phi_u(y)(\tau)\|_X 
&\leq \Big\|\int_0^\tau T_{-1}(\tau-s) B_2\big(f(x(s),u(s)) - f(y(s),u(s))\big)ds\Big\|_X \\
&\leq c_\tau \|f(x(\cdot),u(\cdot)) - f(y(\cdot),u(\cdot))\|_{L^\infty([0,\tau],V)} \\
&\leq c_t L(K)\rho_t(x,y) \\
&\leq \frac{1}{2} \rho_t(x,y),
\end{align*}
for $t \leq t_2$, where $t_2>0$ is a small enough real number, that does not depend on the choice of $x_0 \in W$.

\ifnothabil	\sidenote{\mir{Reference for Full book:\quad	Theorem~\ref{thm:Banach fixed point theorem}.}}\fi

According to the Banach fixed point theorem, there exists a unique solution of $x(t)=\Phi_u(x)(t)$ on $[0,\min\{t_1,t_2\}]$, which is a mild solution of \eqref{eq:SEE+admissible}.

\textbf{General compact $Q$.} Till now, we have shown that for any $w\in Q$, any $r>0$, any bounded set $\Sc \subset U$, and any $\delta>0$, there is a time $t_1 = t_1(w,r,\Sc,\delta)>0$ (that we always take the maximal possible), such that for any $x_0 \in W:=B_r(w)$, and for any $u\in\Uc_{\Sc}$ there is a unique solution of \eqref{eq:SEE+admissible} on $[0,t_1]$, and it lies in the ball $B_{Mr+h_0\|u\|_{L^\infty([0,t_1],U)}+\delta}(w)$.

\amc{It remains to show} that $t_1$ can be chosen uniformly in $w \in Q$, that is\linebreak $\inf_{w \in Q} t_1(w,r,\Sc,\delta)>0$. 
Let this not be so, that is, $\inf_{w \in Q} t_1(w,r,\Sc,\delta)=0$. Then there is a sequence 
$(w_k) \subset Q$, such that the corresponding times $\big(t_1(w_k,r,\Sc,\delta)\big)_{k\in\N}$ monotonically decay to zero.
As $Q$ is compact, there is a converging subsequence of $(w_k)$, converging to some $w^* \in Q$. 
However, $t_1(w^*,r,\Sc,\delta)>0$, which easily leads to a contradiction.
\end{proof}

\begin{remark}
\label{rem:Picard-Lindeloef theorem} 
The technique of proving the \amc{Picard-Lindel\"of theorem is quite classical}. Note however, that here we need to tackle the influence of unbounded input operators, and also we provide a uniform existence time for solutions that controls the maximal deviation of the trajectory from the given set of initial conditions, which is realized by the choice of the spaces $Y_t$ in \eqref{eq:Y_T_Def}. This leads to several changes in the proof of the invariance of $Y_t$ with respect to the operator $\Phi_u(x)$.
\end{remark}

\begin{corollary}[Picard-Lindel\"of theorem for zero-class admissible $B$ and quasi-contractive semigroups]
\label{cor:PL-for zero-class admissible operators and quasi-contractive semigroups} 
Let Assumptions~\ref{ass:Admissibility}, \ref{ass:Integrability}, \ref{Assumption1} hold.
Let also $B$ be zero-class admissible, and $T$ be a quasi-contractive strongly continuous semigroup, that is, there is $\lambda>0$ such that 
\begin{eqnarray}
\|T(t)\| \leq e^{\lambda t},\quad t\geq 0.
\end{eqnarray} 
For any bounded ball $W \subset X$ (with corresponding $w\in X$ and $r>0$: $W= B_r(w)$), any bounded set $\Sc \subset U$, and any $\delta>0$, there is a time $t_1 = t_1(W,\Sc,\delta)>0$, such that for any $x_0 \in W$ and any $u\in\Uc_{\Sc}$ there is a unique solution of \eqref{eq:SEE+admissible} on $[0,t_1]$, and it lies in the ball $B_{r+\delta}(w)$.
\end{corollary}

\begin{proof}
The claim follows directly from Theorem~\ref{PicardCauchy}.
\end{proof}

\begin{remark}
\label{rem:Quasicontractive-SGR-Picard-Lindeloef} 
Without an assumption of quasicontractivity, Corollary~\ref{cor:PL-for zero-class admissible operators and quasi-contractive semigroups} does not hold. 
Consider the special case $f \equiv 0$ and $B\equiv 0$. Then the system \eqref{eq:SEE+admissible} is linear, and for a given $x_0\in X$ the solution of \eqref{eq:SEE+admissible} exists globally and equals $t\mapsto T(t)x_0$. Now take $w:=0$ and pick any $r>0$ and $t_1>0$. Then
\[
\sup_{\tau \in [0,t_1]}\sup_{\|x\|_X\leq r}\|T(\tau)x\|_X= r\sup_{\tau \in [0,t_1]}\|T(\tau)\|.
\]
Since $T$ is merely strongly continuous, the map $t\mapsto \|T(t)\|$ does not have to be continuous at $t=0$, and it may happen that $\lim_{t_1\to 0}\sup_{\tau \in [0,t_1]}\|T(\tau)\| >1$.
\ifExercises\mir{See, e.g., Exercise~\ref{ex:SGR-2019-4.1}. }\fi

Hence, in general, it is not possible to prove that the solution starting at arbitrary $x_0 \in B_r(w)$, will stay in $B_{r+\delta}(w)$ during a sufficiently small and uniform in $x_0 \in B_r(w)$ time. 
\end{remark}

The following example shows that Theorem~\ref{PicardCauchy} does not hold in general if $W$ is a bounded set (and not only a bounded ball over a compact set), even for linear systems governed by contraction semigroups on a Hilbert space.

\begin{example}
\label{examp:PL-difference-to-ODE-case} 
Let $X=\ell_2$, and consider a diagonal semigroup, defined by $T(t)x:=(e^{-kt}x_k)_k$, for all $x=(x_k)_k\in X$ and all $t\geq 0$. This semigroup is strongly continuous and contractive.
Consider a bounded and closed set $W:=\{x \in \ell_2:\ \|x\|_X =1\}$. 
Yet $\|T(t)e_k\|_X = e^{-kt}$, and thus for each $\delta\in(0,1)$ and for each time $t_1>0$, we can find $k\in\N$, such that 
$\|T(t_1)e_k\|_X<1-\delta$, which means that $T(t_1)e_k \notin B_\delta(W)$.
\end{example}

At the same time, a stronger Picard-Lindel\"of-type theorem can be shown for uniformly continuous semigroups (this encompasses, in particular, the case of infinite ODE systems, also called \q{ensembles}), which fully extends the corresponding result for ODE systems, see \cite[Chapter 1]{Mir23}.

\begin{theorem}[Picard-Lindel\"of theorem for uniformly continuous semigroups]
\label{thm:PicardCauchy-Uniformly-cont-semigroups}
Let Assumptions~\ref{ass:Admissibility}, \ref{ass:Integrability}, \ref{Assumption1} hold. 
Let further $T$ be a uniformly continuous semigroup (not necessarily quasicontractive).
For any bounded set $W \subset X$, any bounded set $\Sc \subset U$ and any $\delta>0$, there is a time $\tau = \tau(W,\Sc,\delta)>0$, such that for any $x_0 \in W$, and $u\in\Uc_{\Sc}$ there is a unique solution of \eqref{eq:SEE+admissible} on $[0,\tau]$, and it lies in $B_\delta(W)$.
\end{theorem}

%
%
%
%

\begin{proof}
First note that since $A \in L(X)$, 
for any $a \in \rho(A)$ the norm $x \mapsto \|(aI-A)^{-1}x\|_X$, $x \in X$, is equivalent to the original norm on $X$. Thus $X = X_{-1}$ up to the equivalence of norms. Hence, as $B \in L(U,X_{-1})$, then also $B \in L(U,X)$, and thus, in particular, $B$ is zero-class $\infty$-admissible operator.

Pick any $C>0$ such that $W \subset B_C$, and $\Uc_\Sc \subset B_{C,\Uc}$. 
Take also any $\delta>0$, and consider the following sets (depending on a parameter $t>0$):
\begin{eqnarray}
Y_{t}:= \big\{x \in C([0,t],X): \dist(x(t),W) \leq \delta \ \ \forall t \in [0,t]\big\},
\label{eq:Y_T_Def-uniform}
\end{eqnarray}
endowed with the metric $\rho_{t}(x,y):=\sup_{s \in [0,t]} \|x(s)-y(s)\|_X$, making them complete metric spaces.

Pick any $x_0 \in W$ and any $u \in \Uc_\Sc$. We are going to prove that for small enough $t$, the spaces $Y_{t}$  are invariant under the operator 
$\Phi_u$, defined for any $x \in Y_{t}$ and all $\tau \in [0,t]$ by \eqref{eq:Phi-map-PL-Theorem}.
By Assumptions~\ref{ass:Admissibility}, \ref{ass:Integrability}, the function $\Phi_u(x)$ is continuous.
\ifnothabil	\sidenote{\mir{and by Theorem~\ref{Milde_Loesungen}}}\fi

Fix any $t>0$ and pick any $x \in Y_{t}$.  
Then for any $\tau<t$, it holds that
\begin{align*}
\dist(\Phi_u(x)(\tau),W) &\le \|\Phi_u(x)(\tau)-x_0\|_X \\
&\leq   \Big\|T(\tau)x_0 - x_0\Big\|_X + \Big\|\int_0^\tau T_{-1}(\tau-s) Bu(s)ds\Big\|_X \\
&\qquad\qquad\qquad\qquad\qquad + \Big\|\int_0^\tau T_{-1}(\tau-s) B_2f(x(s),u(s))ds\Big\|_X \\
&\leq   \|T(\tau)-I \|_X \|x_0\|_X +  h_\tau \|u\|_{L^\infty([0,\tau],U)} \\
&\qquad\qquad\qquad\qquad\qquad + c_\tau \|f(x(\cdot),u(\cdot))\|_{L^\infty([0,\tau],V)} \\
			&\leq C \|T(\tau)-I \|_X  + h_\tau \|u\|_{L^\infty([0,\tau],U)}\\
			&\quad+ c_\tau \|f(x(\cdot),u(\cdot)) - f(0,u(\cdot))\|_{L^\infty([0,\tau],V)} + c_\tau \|f(0,u(\cdot))\|_{L^\infty([0,\tau],V)}.
\end{align*}
Now for all $s\in[0,t]$
\begin{align*}
\|x(s)\|_X &\leq C +\delta =: K.
\end{align*}
In view of Assumption~\ref{Assumption1}(iii), it holds that 
\[
\|f(0,u(s))\|_V \leq \sigma(\|u(s)\|_U) + c,\quad \text{ for a.e. } s \in [0,t].
\]
Now Lipschitz continuity of $f$ on bounded balls ensures that there is $L(K)>0$, such that for all $\tau\in[0,t]$ 			
\begin{align*}
\|\Phi_{t}(x)(\tau)-w\|_X 		
			&\leq C \|T(\tau)-I \|_X + h_t \|u\|_{L^\infty([0,t],U)} \\
					&\quad+ c_\tau \big( L(K) \|x\|_{L^\infty([0,t],X)} + \sigma(\|u\|_{L^\infty([0,t],U)}) + c\big)\\
			&\leq C \|T(\tau)-I \|_X + h_t \|u\|_{L^\infty([0,t],U)} \\
					&\quad+ c_t \big( L(K)K + \sigma(C) + c\big).
\end{align*}
Since $T$ is a uniformly continuous semigroup, $h_t\to 0$ as $t\to +0$, and $c_t\to 0$ as $t\to +0$, from this estimate it is clear that  there exists $t_1>0$, depending solely on $C$ and $\delta$, such that 
\[
\dist(\Phi_u(x)(t),W)\le\delta,\quad \text{ for all } t \in [0,t_1].
\]
This means, that $Y_{t}$ is invariant with respect to $\Phi_u$ for all $t \in (0,t_1]$, and $t_1$ does not depend on the choice of $x_0 \in W$.
The rest of the proof is analogous to the proof of Theorem~\ref{PicardCauchy}.
%
%
%
\end{proof}

\subsection{Well-posedness}
\label{sec:Global-Well-posedness}

\index{solution!extension}
Our next aim is to study the prolongations of solutions and their asymptotic properties.

\begin{definition}
Let $x_1(\cdot)$, $x_2(\cdot)$ be mild solutions of \eqref{eq:SEE+admissible} defined on the intervals $[0,t_1)$ and $[0,t_2)$ respectively, $t_1,t_2>0$.
We call $x_2$ an \emph{extension} of $x_1$ if $t_2>t_1$, and $x_2(t)=x_1(t)$ for all $t \in [0,t_1)$.
\end{definition}

\begin{lemma}
\label{lem:Equality-of-solutions} 
Let Assumptions~\ref{ass:Admissibility}, \ref{ass:Integrability}, \ref{Assumption1} hold.
Take any $x_0 \in X$ and $u\in\Uc$. 
Any two solutions of \eqref{eq:SEE+admissible} coincide in their common domain of existence. 
\end{lemma}

The proof is similar to the ODE case \cite[Lemma 1.13]{Mir23} as is omitted.

\index{solution!maximal}
\index{solution!global}
\begin{definition}
A solution $x(\cdot)$ of \eqref{eq:SEE+admissible} is called 
\begin{enumerate}[label=(\roman*)]
	\item \emph{maximal} if there is no solution of \eqref{eq:SEE+admissible} that extends $x(\cdot)$,
	\item \emph{global} if $x(\cdot)$ is defined on $\R_+$.
\end{enumerate}

\end{definition}

A central property of the system \eqref{eq:SEE+admissible} is
\begin{definition}
\label{def:Well-posedness-finite-dim} 
\index{system!well-posed}
We say that the system \eqref{eq:SEE+admissible} is \emph{well-posed} if for every initial value $x_0 \in X$ and every external input $u \in \Uc$, there exists a unique maximal solution $\phi(\cdot,x_0,u):[0,t_m(x_0,u)) \rightarrow X$, where $0 < t_m(x_0,u) \leq \infty$.

We call $t_m(x_0,u)$ the \emph{maximal existence time} of a solution corresponding to $(x_0,u)$.
\end{definition}

\index{flow map}
\index{maximal existence time}

The map $\phi$, defined in Definition~\ref{def:Well-posedness-finite-dim}, and describing the evolution of the system \eqref{eq:SEE+admissible}, 
is called the \emph{flow map}, or just \emph{flow}.
The domain of definition of the flow $\phi$ is 
\[
D_\phi:=\cup_{x_0\in X,\ u\in\Uc}[0,t_m(x_0,u))\times\{(x_0,u)\}.
\] 
In the following pages, we will always deal with maximal solutions. We will usually denote the initial condition by $x \in X$.

\begin{theorem}[Well-posedness]
\label{thm:Global_well-posedness}
Let Assumptions~\ref{ass:Admissibility}, \ref{ass:Integrability}, \ref{Assumption1} hold. 
Then \eqref{eq:SEE+admissible} is well-posed.
\end{theorem}

The proof is similar to the ODE case \cite[Theorem 1.16]{Mir23} as is omitted.

Now we show that well-posed systems \eqref{eq:SEE+admissible} are a special case of general control systems, introduced in Definition~\ref{Steurungssystem}.
\begin{theorem}
\label{thm:SEE-as-control systems} 
Let \eqref{eq:SEE+admissible} be well-posed. Then the triple $(X,\Uc,\phi)$, where $\phi$ is a flow map of \eqref{eq:SEE+admissible}, constitutes a control system in the sense of Definition~\ref{Steurungssystem}.
\end{theorem}

\begin{proof}
\amc{The continuity axiom holds by the definition of a mild solution.
Let us check the cocycle property.} 

Take  any initial condition $x \in X$, any input $u \in \Uc$, and any $t,\tau\geq 0$, such that 
$[0,t+\tau]\tm \{(x,u)\} \subset D_{\phi}$.
Define an input $v$ by $v(r)=u(r+\tau)$, $r \geq 0$.

Due to \eqref{eq:SEE+admissible_Integral_Form}, we have:
\begin{align*}
\phi(t &+ \tau ,x,u) =  T(t+\tau)x + \int_0^{t+\tau} T_{-1}(t+\tau-s) B_2f(\phi(s,x,u),u(s))ds\\
&\qquad\qquad\qquad\qquad\qquad + \int_0^{t+\tau} T_{-1}(t+\tau-s) Bu(s)ds.
\end{align*}
As $T_{-1}(t)$ is a bounded operator, it can be taken out of the Bochner integral:
\begin{align*}
\phi(t &+ \tau ,x,u) =  T(t)T(\tau)x + T_{-1}(t)\int_0^{\tau} T_{-1}(\tau-s) B_2f(\phi(s,x,u),u(s))ds  \\
&\quad+ T_{-1}(t)\int_0^{\tau} T_{-1}(\tau-s) Bu(s)ds \\
&\quad + \int_{\tau}^{t+\tau} T_{-1}(t+\tau-s) B_2f(\phi(s,x,u),u(s))ds + \int_{\tau}^{t+\tau} T_{-1}(t+\tau-s) Bu(s)ds.
\end{align*}
As $B$ is $\infty$-admissible, we have that $\int_0^{\tau} T_{-1}(\tau-s) Bu(s)ds \in X$. 
Since $T_{-1}(\cdot)$ coincides with $T(\cdot)$ on $X$, we infer
\[
T_{-1}(t)\int_0^{\tau} T_{-1}(\tau-s) Bu(s)ds = T(t)\int_0^{\tau} T_{-1}(\tau-s) Bu(s)ds.
\]
Finally,
\begin{align*}
\phi(t + \tau ,x,u) &= T(t) \phi(\tau,x,u) + \int_{0}^{t} T_{-1}(t-s) B_2 f(\phi(s+\tau,x,v),v(s))ds \\
&\qquad\qquad\qquad\qquad\qquad + \int_{0}^{t} T_{-1}(t-s) Bv(s)ds\\
 & =  \phi(t,\phi(\tau,x,u),v),
\end{align*}
and the cocycle property holds. 
\amc{The rest of the properties of control systems are fulfilled by construction.}
\end{proof}


\amc{We proceed with a proof of a boundedness-implies-continuation property.}
\begin{proposition}
\label{prop:Unboundedness_finite_existence_time}
Let Assumptions~\ref{ass:Admissibility}, \ref{ass:Integrability}, \ref{Assumption1} hold. 
Then \eqref{eq:SEE+admissible} has \amc{the} BIC property. 
\end{proposition}

\begin{proof}
\amc{Pick any $x \in X$, any $u \in \Uc$,} and consider the corresponding maximal solution $\phi(\cdot,x,u)$, defined on $[0,t_m(x,u))$.
Assume that $t_m(x,u)<+\infty$, but at the same time $\Liminf_{t\to t_m(x,u)-0}\|\phi(t,x,u)\|_X <\infty$.
Then there is a sequence $(t_k)$, such that $t_k \to t_m(x,u)$ as $k\to\infty$ and 
$\lim_{k\to\infty}\|\phi(t_k,x,u)\|_X <\infty$. \amc{Hence,} also $\sup_{k\in\N}\|\phi(t_k,x,u)\|_X =:C<\infty$.

Let $\tau(C)>0$ be a uniform existence time for the solutions starting in the ball $\clo{B_C}$ subject to inputs of a magnitude not exceeding $\|u\|$, which exists and is positive in view of  Theorem~\ref{PicardCauchy}.
Then the solution of \eqref{eq:SEE+admissible} starting in $\phi(t_k,x,u)$, corresponding to the input $u(\cdot + t_k)$, exists and is unique on $[0,\tau(C)]$ by Theorem~\ref{PicardCauchy}, and by the \amc{cocycle property, $\phi(\cdot,x,u)$ can be prolonged} to $[0,t_k +\tau(C))$, which (since $t_k\to t_m(x,u)$ as $k\to\infty$) contradicts to the maximality of the solution corresponding to $(x,u)$.

Hence $\Liminf_{t\to t_m(x,u)-0}\|\phi(t,x,u)\|_X =\infty$, which implies the claim.
\end{proof}

\subsection{Forward completeness and boundedness of reachability sets}
\label{sec:Boundedness of reachability sets}


Local Lipschitz continuity guarantees \amc{the local existence of solutions}. To ensure the global existence of solutions, stronger requirements on nonlinearity are needed.
\begin{proposition}
\label{prop:Unif-Glob-Lip-and-BRS}
Let Assumptions~\ref{ass:Admissibility}, \ref{ass:Integrability}, \ref{Assumption1} hold.
Let further $f$ be uniformly globally Lipschitz.
Then \eqref{eq:SEE+admissible} is forward complete and has BRS.
\end{proposition}

\begin{proof}
By Theorem~\ref{PicardCauchy}, \amc{for any $x_0 \in X$ and any $u \in \Uc$} there exists a mild solution of \eqref{eq:SEE+admissible}, with a maximal existence time $t_m (x_0, u)$, which may be finite or infinite. Let $t_m (x_0, u)$ be finite. 

Let $L>0$ be a uniform global Lipschitz constant for $f$. 
As $\|T(t)\|\leq Me^{\lambda t}$ for some $M\ge 1$, $\lambda\geq 0$ and all $t\geq 0$, for any $t<t_m (x_0, u)$ we have according to the formula \eqref{eq:SEE+admissible_Integral_Form} the following estimates
\begin{align*}
\|\phi(&t,x_0,u)\|_X \leq  \|T(t)\| \|x_0\|_X + \Big\|\int_{0}^{t} T_{-1}(t-s) Bu(s)ds\Big\|_X \\
&\qquad\qquad\qquad\qquad + \Big\|\int_0^t T_{-1}(t-s)B_2 f(\phi(s,x_0,u),u(s))ds\Big\|_X  \\
			&\leq Me^{\lambda t} \|x_0\|_X + h_t \|u\|_{\Uc} + c_t \|f(\phi(\cdot,x_0,u),u(\cdot))\|_{L^\infty([0,t],V)}\\
			&\leq Me^{\lambda t} \|x_0\|_X + h_t \|u\|_{\Uc} 
			+ c_t \|f(\phi(\cdot,x_0,u),u(\cdot)) - f(0, u(\cdot))\|_{L^\infty([0,t],V)}\\
			&\qquad\qquad\qquad\qquad+ c_t \|f(0,u(\cdot))\|_{L^\infty([0,t],V)}\\
			&\leq Me^{\lambda t} \|x_0\|_X + h_t \|u\|_{\Uc} 
			+ c_t L \|\phi(\cdot,x_0,u)\|_{L^\infty([0,t],X)} + c_t (\sigma(\|u\|_\Uc) + c).
\end{align*}
\amc{Since $c_t \to 0$ whenever $t \to +0$,} there is some $t_1 \in (0,t_m (x_0, u))$ such that $c_{t_1}L\leq\frac{1}{2}$. Then it holds that 
\begin{align}
\label{eq:BRS-estimate}
\sup_{t \in[0,t_1]}\|\phi(t,x_0,u)\|_X 
			\leq 2\Big(Me^{\lambda t_1} \|x_0\|_X + h_{t_1} \|u\|_{\Uc} + c_{t_1} (\sigma(\|u\|_\Uc) + c)\Big).
\end{align}
Note that $t_1$ does not depend on $x_0$ and $u$. Hence, using cocycle property and with $\phi(t_1,x_0,u)$ instead of $x_0$, we obtain a uniform bound for $\phi(\cdot,x_0,u)$ on $2t_1$, $3t_1$, and so on. Thus, $\phi(\cdot,x_0,u)$ 
is uniformly bounded on $[0,t_m (x_0, u))$, and hence can be prolonged to a larger interval by \amc{the} BIC property ensured by Proposition~\ref{prop:Unboundedness_finite_existence_time}, a contradiction to the definition of $t_m (x_0, u)$. \amc{Overall,} $\Sigma$ is forward complete, and the estimate 
\eqref{eq:BRS-estimate} iterated as above to larger intervals shows that  \eqref{eq:SEE+admissible} has BRS.
%
%
\end{proof}

\subsection{Regularity of the flow map}

We start \amc{this section} with a basic result describing the exponential deviation between two trajectories.
\begin{theorem}
\label{thm:Deviation-of-trajectories}
Let Assumptions~\ref{ass:Admissibility}, \ref{ass:Integrability}, \ref{Assumption1} hold.
Take $M \ge 1$, $\lambda\geq 0$  such that $\|T(t)\|\leq Me^{\lambda t}$ for all $t\geq 0$.
Pick any $x_1,x_2 \in X$, any $u\in\Uc$, and let $\phi(\cdot,x_1,u)$ and $\phi(\cdot,x_2,u)$ be defined on a certain common interval $[0,\tau]$.

Then there exists $R=R(x_1,x_2,\tau,u)>0$, such that 
\begin{align}
\label{eq:Exponential-deviation}
\|\phi(t,x_1,u) - \phi(t,x_2,u)\|_X  \leq  2M\|x_1-x_2\|_X e^{Rt},\quad t\in[0,\tau].
\end{align}
\end{theorem}

\begin{proof}
Pick any $x_1,x_2 \in X$, any $u \in \Uc$, and let $\phi_i(t):=\phi(t,x_i,u)$, $i=1,2$ be the corresponding (unique) maximal solutions of \eqref{eq:SEE+admissible} (guaranteed by Theorem~\ref{thm:Global_well-posedness}), defined on $[0,\tau]$, for a certain $\tau>0$.

Set 
\[
K:=\max\big\{\sup_{0 \leq t \leq \tau}\|\phi_1(t)\|_X,\sup_{0 \leq t \leq  \tau}\|\phi_2(t)\|_X,\|u\|_\Uc\big\}<\infty,
\]
 where $K$ is finite due to \amc{the} continuity of trajectories.

Due to  \eqref{eq:SEE+admissible_Integral_Form}, and using Lipschitz continuity of $f$ (see \eqref{eq:Lipschitz}), we have for any $t \in [0,\tau]$:
\begin{align*}
\|\phi_1(t) - \phi_2(t)\|_X &\leq \|T(t)\| \|x_1-x_2\|_X  \\
&\qquad+ \Big\|\int_0^t T_{-1}(t-s) B_2 \Big(f(\phi_1(s),u(s)) - f(\phi_2(s),u(s))\Big)ds\Big\|_X \\
&\leq Me^{\lambda t} \|x_1-x_2\|_X  +  c_t \|f(\phi_1(\cdot),u) - f(\phi_2(\cdot),u)\|_{L^\infty([0,t],X)}\\
&\leq Me^{\lambda t} \|x_1-x_2\|_X  +  c_t L(K) \|\phi_1(\cdot)-\phi_2(\cdot)\|_{L^\infty([0,t],X)}.
\end{align*}
As $c_t \to 0$ as $t \to +0$, there is some $t_1 \in (0,\tau)$ such that $c_{t_1}L(K)\leq\frac{1}{2}$. 
Note that $t_1$ depends on $\tau$ only (as $K$ does).

 Then, taking the supremum of the previous expression over $[0,t]$, with $t<t_1$, we obtain that 
\begin{align*}
\|\phi_1(t) - \phi_2(t)\|_X 
&\leq 2Me^{\lambda t} \|x_1-x_2\|_X,\quad t\in[0,t_1].
\end{align*}
Take $k\in\N$ such that $kt_1 <\tau$ and $(k+1)t_1>\tau$. 
Then, using the cocycle property, for any $l\in\N$, $l\leq k$ and all $t\in [0,t_1]$ s.t. $lt_1+t<\tau$ we have
\begin{align*}
\|\phi_1(lt_1 + t) - \phi_2(lt_1+t)\|_X 
&\leq (2M)^{l+1} e^{\lambda(lt_1 + t)} \|x_1-x_2\|_X\\
&= 2M e^{l \ln(2M) + \lambda lt_1 + \lambda t} \|x_1-x_2\|_X\\
[R:=\lambda + \frac{1}{t_1}\ln(2M)>\lambda] &= 2M e^{l Rt_1 + \lambda t} \|x_1-x_2\|_X\\
&\le 2M e^{R (lt_1 + t)} \|x_1-x_2\|_X.
\end{align*}
This shows \eqref{eq:Exponential-deviation}.
\end{proof}

%
%

\begin{definition}
\label{axiom:Lipschitz}
\emph{The flow of a forward complete control system $\Sigma=(X,\Uc,\phi)$, is called Lipschitz continuous on compact intervals (for uniformly bounded inputs)}, if 
for any $\tau>0$ and any $C>0$ there exists $L>0$ so that for any $x_1,x_2 \in \clo{B_C}$,
for all $u \in B_{C,\Uc}$, it holds that 
\begin{eqnarray}
\|\phi(t,x_1,u) - \phi(t,x_2,u) \|_X \leq L \|x_1-x_2\|_X,\quad t \in [0,\tau].
\label{eq:Flow_is_Lipschitz}
\end{eqnarray} 
\end{definition}

Theorem~\ref{thm:Deviation-of-trajectories} estimates the deviation between two trajectories. To have a stronger result, showing the Lipschitz continuity of the flow map $\phi$, we additionally assume the BRS property of \eqref{eq:SEE+admissible}.
\begin{theorem}
\label{thm:Lipschitz-continuity-of-flow}
\amc{Suppose that Assumptions~\ref{ass:Admissibility}, \ref{ass:Integrability}, \ref{Assumption1} hold and \eqref{eq:SEE+admissible} has BRS.
 Then the flow of \eqref{eq:SEE+admissible} is Lipschitz continuous on compact intervals for uniformly bounded inputs.}
\end{theorem}

\begin{proof}
Take any $C>0$ and pick any $x_1,x_2 \in B_C$, and any $u \in \Uc$ with
$\|u\|_{\Uc} \leq C$.
Let $\phi_i(\cdot):=\phi(\cdot,x_i,u)$, $i=1,2$ be the corresponding maximal solutions of \eqref{eq:SEE+admissible}. These solutions are global since we assume that \eqref{eq:SEE+admissible} is forward-complete.

As \eqref{eq:SEE+admissible} is BRS, the following quantity is finite for any $\tau>0$:
\[
K(\tau):=\sup_{t\in[0,\tau],\ x\in B_C,\ u\in B_{C,\Uc}}\|\phi(t,x,u)\|_X<\infty.
\]
Following the lines of the proof of Theorem~\ref{thm:Deviation-of-trajectories}, we obtain the claim.
%
%
%
\end{proof}

\begin{definition}
\label{def:Continuous-dependence-systems} 
Let $\Sigma=(X,\Uc,\phi)$ be a forward complete control system. 
We say that the flow
\emph{$\phi$ depends continuously on inputs and on initial states}, if 
 for all $x \in X$, $u \in \Uc$, $\tau>0$, and all $\eps>0$ there exists $\delta>0$, such that $\forall x' \in X: \|x-x'\|_X< \delta$ and
$\forall u' \in \Uc: \|u-u'\|_{\Uc}< \delta$ it holds that
\[
\|\phi(t,x,u)-\phi(t,x',u')\|_X< \eps, \quad t \in [0,\tau].
\]
\end{definition}

To obtain \amc{the} continuity of the flow map with respect to both states and inputs, which is important for the application of the density argument, we impose additional conditions on the nonlinearity $f$.
\begin{theorem}
\label{Continuous_dependence_Thm}
Let Assumptions~\ref{ass:Admissibility}, \ref{ass:Integrability}, \ref{Assumption1} hold.
Let further there exists $q \in \Kinf$ such that for all $C>0$ there is $L(C)>0$: 
for all $x_1,x_2 \in \clo{B_C}$ and all $v_1,v_2 \in \clo{B_{C,U}}$ it holds that
\begin{equation}
\|f(x_1,v_1)-f(x_2,v_2)\|_V \leq L(C) \big( \|x_1-x_2\|_X + q(\|v_1-v_2\|_U) \big).
\label{eq:Lipschitz_On_X_U}
\end{equation}
If \eqref{eq:SEE+admissible} has the BRS property, then the flow of \eqref{eq:SEE+admissible} depends continuously on initial states and inputs.
\end{theorem}

\begin{proof}
Pick any time $\tau>0$. 
\amc{Take any $C>0$, any $x_1,x_2 \in \clo{B_C}$, and any $u_1,u_2 \in \clo{B_{C,\Uc}}$.} Let $\phi_i(\cdot)=\phi(\cdot,x_i,u_i)$, $i=1,2$ be the corresponding global solutions.

Due to  \eqref{eq:SEE+admissible_Integral_Form}, we have:
\begin{align*}
\|\phi_1(t) - \phi_2(t)\|_X &\leq \|T(t)\| \|x_1-x_2\|_X  +h_t \|u_1-u_2\|_{\Uc}\\
&\qquad + c_t \sup_{r\in[0,t]}\big\|f(\phi_1(r),u_1(r))-f(\phi_2(r),u_2(r))\big\|_V.
\end{align*}

In view of the boundedness of reachability sets for the system \eqref{eq:SEE+admissible}, we have
\[
K:=\sup_{\|z\|_X\leq C,\ \|u\|_\Uc\leq C,\ t \in [0,\tau]}\|\phi(t,z,u)\|_X < \infty.
\]
As $\|T(t)\|\leq Me^{\lambda t}$ for some $M,\lambda\geq 0$ and all $t\geq 0$, and due to the property 
\eqref{eq:Lipschitz_On_X_U} with $L:=L(K)$ (note that $K\geq C$), we can continue above estimates to obtain
\begin{align*}
\|\phi_1(t) &- \phi_2(t)\|_X \leq Me^{\lambda t} \|x_1-x_2\|_X  + h_t \|u_1-u_2\|_{\Uc} \\
&\qquad + c_t L(K) \sup_{r\in[0,t]} \Big( \|\phi_1(r)-\phi_2(r)\|_X + q\big(\|u_1(r)-u_2(r)\|_U\big)\Big).
\end{align*}
\amc{Since} $B$ is zero-class admissible, there is $t_1>0$ such that $c_{t_1} L(K) = \frac{1}{2}$, and thus taking the supremum of both sides over $t\in[0,t_1]$, we have for all $t\in[0,t_1]$ that
\begin{align*}
\|\phi_1(t) &- \phi_2(t)\|_X \leq 2Me^{\lambda t_1} \|x_1-x_2\|_X  + 2h_{t_1} \|u_1-u_2\|_{\Uc} + q\big(\|u_1-u_2\|_\Uc\big).
\end{align*}
Thus, for each $\varepsilon>0$ there is $\delta>0$ so that for all $x_2 \in B_\delta(x_1)$ and for all $u_2 \in B_{\delta,\Uc}(u_1)$ it holds that 
\begin{align*}
\|\phi_1(t) &- \phi_2(t)\|_X \leq \varepsilon,\quad t\in[0,t_1].
\end{align*}
This establishes the continuity over the interval $[0,t_1]$. To obtain continuity over the interval $[0,\tau]$, one can follow the strategy in the second part of the proof of Lemma~\ref{lem:RobustEquilibriumPoint} (and noting that at all the steps the parameter $K$ does not change).
\end{proof}

\subsection{Continuity at trivial equilibrium}

Without loss of generality, we restrict our analysis to fixed points of the form $(0,0) \in X \times
\Uc$. Note that $(0,0)$ is in $X \times \Uc$ since both $X$ and $\Uc$ are linear spaces.

To describe the behavior of solutions near the equilibrium, the following notion is of importance:
\begin{definition}
\label{def:RobustEquilibrium_Undisturbed}
\index{property!CEP}
Consider a system $\Sigma=(X,\Uc,\phi)$ with equilibrium point $0\in X$.
We say that
\emph{$\phi$ is continuous at the equilibrium} if for every $\eps >0$ and for any $h>0$ there exists a $\delta =
          \delta (\eps,h)>0$, so that $[0,h] \tm B_\delta \tm B_{\delta,\Uc} \subset D_\phi$, and
\begin{eqnarray}
 t\in[0,h],\ \|x\|_X \leq \delta,\ \|u\|_{\Uc} \leq \delta \qrq  \|\phi(t,x,u)\|_X \leq \eps.
\label{eq:RobEqPoint}
\end{eqnarray}
In this case, we will also say that $\Sigma$ has the \emph{CEP property}.
\end{definition}

CEP property is a \q{local in time version} of Lyapunov stability and is important, in particular, for the ISS superposition theorems \cite{MiW18b} and for the applications of the non-coercive ISS Lyapunov theory \cite{JMP20}.
\begin{lemma}[Continuity at equilibrium for \eqref{eq:SEE+admissible}]
\label{lem:RobustEquilibriumPoint}
Let Assumptions~\ref{ass:Admissibility}, \ref{ass:Integrability}, \ref{Assumption1} hold, and let $f(0,0)=0$.
Then $\phi$ is continuous at the equilibrium.
\end{lemma}

\begin{proof}
Consider the following auxiliary system
\begin{subequations}
\label{xdot=Ax+f_xu_saturated}
\begin{eqnarray}
\dot{x}(t) & = & Ax(t) + B_2\tilde{f}(x(t),u(t)) + Bu(t),\quad t>0, \\
x(0)  &=&  x_0,
\end{eqnarray}
\end{subequations}
where 
\[
\tilde{f}(x,u):=f\big(\sat(x),\sat_2(u)\big),\quad x\in X,\ u\in U, 
\]
and the saturation function is given for the vectors $z$ in $X$ and in $U$ respectively by
\[
\sat(z):=
 \begin{cases}
           z, &   \|z\|_X\leq 1, \\
           \frac{z}{\|z\|_X}, &   \text{otherwise},
 \end{cases}
\qquad
\sat_2(z):=
 \begin{cases}
           z, &   \|z\|_U\leq 1, \\
           \frac{z}{\|z\|_U}, &   \text{otherwise}. 
 \end{cases} 
\]
As $f$ satisfies Assumption~\ref{Assumption1}, one can show that $\tilde{f}$ is uniformly globally Lipschitz continuous. Hence, \eqref{xdot=Ax+f_xu_saturated} is forward complete and has BRS property by Proposition~\ref{prop:Unif-Glob-Lip-and-BRS}.

We denote the flow of \eqref{xdot=Ax+f_xu_saturated} by $\tilde{\phi} = \tilde{\phi}(t,x,u)$. As $f(x,u)=\tilde{f}(x,u)$ whenever $\|x\|_X\leq 1$ and $\|u\|_U\leq 1$,  it holds also 
\[
\phi(t,x,u) = \tilde{\phi}(t,x,u),
\]
provided that $\|u\|_{\Uc}\leq 1$, $\phi(\cdot,x,u)$ exists on $[0,t]$, and $\|\phi(s,x,u)\|_X\leq 1$ for all $s\in[0,t]$.

\amc{Pick any $\eps \in (0,1)$, $\tau \geq 0$, $\delta \in (0,\varepsilon)$, $x \in B_\delta$, and any $u\in B_{\delta,\Uc}$.} It holds that
\begin{eqnarray*}
\|\tilde{\phi}(t,x,u) \|_X 
            &\leq& \|\tilde{\phi}(t,x,u) -\tilde{\phi}(t,0,u)\|_X + \|\tilde{\phi}(t,0,u)\|_X.
\end{eqnarray*}
Since \eqref{xdot=Ax+f_xu_saturated} has BRS property, by Theorem~\ref{thm:Lipschitz-continuity-of-flow}, the flow of \eqref{xdot=Ax+f_xu_saturated} is
Lipschitz continuous on compact time intervals. Hence there exists a $L(\tau,\delta)>0$ so that for all $t\in[0,\tau]$
\begin{eqnarray}
\label{eq:CEP-estimate-1}
\|\tilde{\phi}(t,x,u) -\tilde{\phi}(t,0,u)\|_X \leq L(\tau,\delta)\|x\|_X \leq L(\tau,\delta)\delta.
\end{eqnarray}

Let us estimate $\|\tilde{\phi}(t,0,u)\|_X$. We have:
\begin{align*}
\|\tilde{\phi}(t,0,u)\|_X 
&\le \Big\|\int_0^t T_{-1}(t-s) B_2\tilde{f}\big(\tilde{\phi}(s,0,u),u(s)\big)ds\Big\|_X + \Big\|\int_0^t T_{-1}(t-s)Bu(s)ds\Big\|_X\\
&\leq c_t\esssup_{s\in[0,t]}\big\|\tilde{f}\big(\tilde{\phi}(s,0,u),u(s)\big)\big\|_X + h_t\big\|u\big\|_{L^\infty([0,t],U)}\\
&\leq c_t\esssup_{s\in[0,t]}\big\|\tilde{f}\big(\tilde{\phi}(s,0,u), u(s)\big) - \tilde{f}\big(0, u(s)\big)\big\|_X + c_t\esssup_{s\in[0,t]}\big\| \tilde{f}\big(0, u(s)\big)\big\|_X \\
&\qquad\qquad\qquad + h_t\big\|u\big\|_{L^\infty([0,t],U)}.
\end{align*}
Since $\tilde{f}(0,\cdot)$ is continuous, for any $\eps_2>0$ there exists $\delta_2<\delta$ so that $u(s) \in B_{\delta_2}$ implies that 
$\|\tilde{f}(0,u(s))-\tilde{f}(0,0)\|_X \leq \eps_2$. Since $\tilde{f}(0,0)=0$, for the above $u$ we have
$\|\tilde{f}(0,u(s))\|_X \leq \eps_2$.

As $\tilde{f}$ is uniformly globally Lipschitz, there is $L>0$ such that for the inputs satisfying $\|u\|_\Uc \leq \delta_2$ we have  
\begin{align*}
\|\tilde{\phi}(t,0,u)\|_X \leq c_tL\esssup_{s\in[0,t]}\big\|\tilde{\phi}(s,0,u)\big\|_X + c_t \varepsilon_2 + h_t\delta_2.
\end{align*}
As $c_t\to 0$ for $t\to+0$, there is $t_1>0$, such that $c_{t_1}L\leq \frac{1}{2}$.
 
Then we have that
\begin{align}
\label{eq:CEP-estimate-2}
\|\tilde{\phi}(t,0,u)\|_X \leq 2c_{t_1} \varepsilon_2 + 2h_{t_1}\delta_2,\quad t \leq t_1.
\end{align}

Combining \eqref{eq:CEP-estimate-1} with \eqref{eq:CEP-estimate-2}, we see that whenever $\|x\|_X\leq \delta_2$ and $\|u\|_\Uc \leq \delta_2$, it holds that 
\begin{eqnarray*}
\|\tilde{\phi}(t,x,u) \|_X 
             \leq L(\tau,\delta_2)\delta_2 + 2c_{t_1} \varepsilon_2 + 2h_{t_1}\delta_2,\quad t \leq t_1.
\end{eqnarray*}
Now for any $\varepsilon<1$ we can find $\delta_2<\varepsilon$, such that 
\begin{eqnarray*}
\|\tilde{\phi}(t,x,u) \|_X 
             \leq  \varepsilon,\quad t \leq t_1,\quad \|x\|_X\leq \delta_2,\quad \|u\|_\Uc \leq \delta_2.
\end{eqnarray*}
As $\tilde{\phi}(t,x,u) = \phi(t,x,u)$ whenever $\|\tilde{\phi}(t,x,u)\|_X<1$, we obtain that 
\begin{eqnarray*}
\|\phi(t,x,u) \|_X  \leq  \varepsilon,\quad t \leq t_1,\quad \|x\|_X\leq \delta_2,\quad \|u\|_\Uc \leq \delta_2.
\end{eqnarray*}
Note that $t_1$ depends on $L$ only, and does not depend on $\delta_2$. 
Thus, one can find $\delta_3<\delta_2$, such that
\begin{eqnarray*}
\|\phi(t,x,u) \|_X  \leq  \delta_2,\quad t \leq t_1,\ \|x\|_X\leq \delta_3,\ \|u\|_\Uc \leq \delta_3.
\end{eqnarray*}
By \amc{the} cocycle property, we obtain that 
\begin{eqnarray*}
\|\phi(t,x,u) \|_X \leq  \varepsilon,\quad t \leq 2t_1,\ \|x\|_X\leq \delta_3,\ \|u\|_\Uc \leq \delta_3.
\end{eqnarray*}
Iterating this procedure, we obtain that there is some $\omega>0$, such that 
\begin{eqnarray*}
\|\phi(t,x,u) \|_X \leq \varepsilon,\quad t \in [0,\tau],\ \|x\|_X\leq \omega,\ \|u\|_\Uc \leq \omega.
\end{eqnarray*}
This shows the CEP property.
\end{proof}


\section{Semilinear analytic systems}
\label{sec:Semilinear analytic boundary control systems}

\subsection{Preliminaries for analytic semigroups and admissibility}
\label{sec:Preliminaries for analytic semigroups}

Recall that $\omega_0(T)$ denotes the growth bound of the semigroup $T$. 
Pick any $\omega>\omega_0(T)$ and define the space $X_\alpha$ and \amc{the norm in it as} 
\begin{eqnarray}
X_\alpha:=D((\omega I-A)^\alpha),\quad \|x\|_{X_\alpha}:=\|(\omega I-A)^\alpha x\|_X,\quad x \in X_\alpha.
\label{eq:X-power-alpha}
\end{eqnarray} 
Furthermore, define the spaces $X_{-\alpha}$ as the completion of $X$ with respect to the norm $x \mapsto \|(\omega I-A)^{-\alpha}x\|_X$. 
For the theory of fractional powers of operators and fractional spaces, see \cite{Haase06} and \cite[Section 1.4]{Hen81} and, for a very brief description of the essentials required here, \cite{Sch20}.

The following well-known property holds:
\begin{proposition}
\label{prop:Analytic-semigroup-Bound-on-A^alphaT(t)} 
Let $T$ be an analytic semigroup on a Banach space $X$ with the generator $A$.
\amc{Then for each $\omega,\kappa>\omega_0(T)$, each $\alpha \in[0,1)$, and each $t>0$} we have $\im(T(t)) \subset X_\alpha$, and there is $C_\alpha>0$ such that 
\begin{eqnarray}
\|(\omega I-A)^\alpha T(t)\|  \leq \frac{C_\alpha}{t^\alpha}e^{\kappa t},\quad t>0.
\label{eq:Bound-on-A^alphaT(t)}
\end{eqnarray}
Furthermore, the map $t\mapsto (\omega I-A)^\alpha T(t)$ is continuous on $(0,+\infty)$ in the uniform operator topology.
\end{proposition}

Next, we formulate a sufficient condition for the zero-class admissibility of input operators for analytic systems. 
Part (ii) of the following proposition is (up to the zero-class statement) contained in \cite[Proposition 2.13]{Sch20}.
We however provide a short proof based on the statement (i) to be self-contained.
\begin{proposition}
\label{prop:ISS-analytic-systems} 
Assume that $A$ generates an analytic semigroup $T$ and $B\in L(U,X_{-1+\alpha})$ for some $\alpha \in (0,1)$. 
Then:

\begin{enumerate}[label = (\roman*)]
	\item For any $\omega>\omega_0(T)$, any $d \in[0,1)$ the operator $(\omega I-A_{-1})^{d}$ is zero-class $p$-admissible for any $p\in (\frac{1}{1-d},+\infty]$. 
	In particular, for any $g \in L^p_{\loc}(\R_+,X)$, the following map 
\begin{eqnarray}
\amc{\xi:t \mapsto \int_0^t (\omega I-A)^d T(t-s)g(s)ds = \int_0^t T_{-1}(t-s)(\omega I-A)^d g(s)ds}
\label{eq:Continuity-mild-solutions-analytic-extension}
\end{eqnarray}
is well-defined and continuous on $\R_+$.

Furthermore, for any $\kappa>\omega_0(T)$ there is $R=R(\kappa,d)$ such that for any $g \in L^\infty_{\loc}(\R_+,X)$ the following holds:
\begin{align}
\int_0^t \big\|(\omega I-A)^d  T(t-s) g(s)\big\|_X ds  \leq R t^{1-d} e^{\kappa t} \| g\|_{L^\infty([0,t],X)}.
\label{eq:Convolution-analytic-for-Linfty-norm}
\end{align}

	\item $B$ is zero-class $q$-admissible for $q\in(\frac{1}{\alpha},+\infty]$.
	\item For any $\omega>\omega_0(T)$, any $d \in[0,\alpha)$ the operator $(\omega I-A)^{d}B$ is zero-class $\infty$-admissible.
\end{enumerate}
%
%
%
\end{proposition}

\begin{proof}
\textbf{(i).} 
Since $T$ is an analytic semigroup, $T(t)$ maps $X$ to $D(A)$ for any $t>0$. \amc{As $D(A)\subset X_d$} for all $d\in[0,1]$, 
the integrand in \eqref{eq:Continuity-mild-solutions-analytic-extension} is in $X$ for a.e. $s\in[0,t)$.
Let us show \amc{the} Bochner integrability of $X$-valued map $s \mapsto (\omega I-A)^d T(t-s)g(s)$ on $[0,t]$.

As $g \in L^{1}_{\loc}(\R_+,X)$, by the criterion of Bochner integrability, $g$ is strongly measurable and $\int_I \|g(s)\|_X ds < \infty$ for any bounded interval $I \subset\R_+$.

Denote by $\chi_{\Omega}$ the characteristic function of the set $\Omega\subset \R_+$. Recall that the map $t\mapsto (\omega I-A)^d T(t)$ is continuous outside of $t=0$ in view of Proposition~\ref{prop:Analytic-semigroup-Bound-on-A^alphaT(t)}. 

If $g(s) = \chi_{\Omega}(s)x$ for some measurable $\Omega \subset \R_+$ and $x \in X$, then the function
\[
s\mapsto (\omega I-A)^d T(t-s)g(s) = (\omega I-A)^d T(t-s)\chi_{\Omega}(s)x
\]
is measurable as a product of a measurable scalar function and a continuous (and thus measurable) vector-valued function.
By linearity, $s\mapsto (\omega I-A)^d T(t-s)g(s)$ is strongly measurable if $g$ is a simple function (see \cite[Section 1.1]{ABH11} for definitions). 

As $g$ is strongly measurable, there is a sequence of simple functions $(g_n)_{n\in\N}$, converging pointwise to $g$ almost everywhere. 
Consider a sequence 
\begin{eqnarray}
\big(s\mapsto (\omega I-A)^d T(t-s)g_n(s)\big)_{n\in\N}
\label{eq:SimpleFun-Seq}
\end{eqnarray}
and take any $s \in [0,t)$ such that $g_n(s) \to g(s)$ as $n\to\infty$.
We have that 
\begin{align*}
\big\|(\omega I-A)^d T(t-s)&g_n(s)- (\omega I-A)^d T(t-s)g(s)\big\|_X \\
&\leq \|(\omega I-A)^d T(t-s)\| \|g_n(s)-g(s)\|_X \to 0,\quad n\to\infty.
\end{align*}
Hence a sequence of strongly measurable functions \eqref{eq:SimpleFun-Seq} converges a.e. to $s\mapsto (\omega I-A)^d T(t-s)g(s)$,
and thus $s\mapsto (\omega I-A)^d T(t-s)g(s)$ is strongly measurable by \cite[Corollary 1.1.2]{ABH11}.

Furthermore, for any $t >0$, using Proposition~\ref{prop:Analytic-semigroup-Bound-on-A^alphaT(t)}, we have that for any $\kappa>\omega_0(T)$ there is $C_d>0$ such that
\begin{align}
\int_0^t \big\|(\omega I-A)^d  T(t-s) g(s)\big\|_X ds  
&\leq \int_0^t \|(\omega I-A)^d T(t-s)\| \| g(s)\|_X ds \nonumber \\
&\leq \int_0^t \frac{C_d}{(t-s)^d}e^{\kappa (t-s)} \| g(s)\|_X ds \nonumber\\
&\leq C_d e^{\kappa t}\int_0^t \frac{1}{(t-s)^d} \| g(s)\|_X ds.
\label{eq:Convolution-boundedness-for-Linfty-norm}
\end{align}

Using H\"older's inequality with a finite $p>\frac{1}{1-d}$, we obtain
\begin{align}
\int_0^t \|(\omega I-A)^d  T(t-s) g(s)\|_X ds  
&\leq C_d e^{\kappa t} \Big(\int_0^t \Big(\frac{1}{(t-s)^d}\Big)^b ds\Big)^{\frac{1}{b}}  \Big( \int_0^t \| g(s)\|^p_X ds\Big)^{\frac{1}{p}}\nonumber\\
&\leq \frac{C_d}{(1-d b)^{1/b}} e^{\kappa t} t^{\frac{1-d b}{b}}  \Big( \int_0^t \| g(s)\|^p_X ds\Big)^{\frac{1}{p}},
\label{eq:Convolution-boundedness}
\end{align}
where $\frac{1}{b} + \frac{1}{p} = 1$, and thus $b$ satisfies $b<\frac{1}{d}$.

Finally, by \cite[Theorem 1.1.4]{ABH11}, 
the map $s \mapsto (\omega I-A)^d  T(t-s)g(s)$ is Bochner integrable on each $[0,t]\subset \R_+$. 
This shows $p$-admissibility of $(\omega I-A)^d$, and if $p<+\infty$, \eqref{eq:Convolution-boundedness} implies zero-class $p$-admissibility of 
$(\omega I-A)^d$. Continuity of the map $\xi$ follows from \cite[Proposition 2.3]{Wei89b}.

For the last claim of item (i), we take $g\in L^\infty_{\loc}(\R_+,X)$ and continue the estimates in \eqref{eq:Convolution-boundedness-for-Linfty-norm} as follows:
\begin{align*}
\int_0^t \big\|(\omega I-A)^d  T(t-s) g(s)\big\|_X ds  
&\leq C_d e^{\kappa t} \int_0^t \frac{1}{(t-s)^d} ds  \| g\|_{L^\infty([0,t],X)},
\end{align*}
and \eqref{eq:Convolution-analytic-for-Linfty-norm} holds with $R=\frac{C_d}{1-d}$.
This implies zero-class $\infty$-admissibility of $(\omega I-A)^d$.

\textbf{(ii).} Take any $\omega>\omega_0(T)$, and 
consider the corresponding norm on $X_{-1+\alpha}$:
\begin{align*}
\|B\|_{L(U,X_{-1+\alpha})} 
&= \sup_{u\in U:\|u\|_U=1}\|Bu\|_{X_{-1+\alpha}}\\
&= \sup_{u\in U:\|u\|_U=1}\|(\omega I-A)^{-1+\alpha}Bu\|_{X} = \|(\omega I-A)^{-1+\alpha}B\|_{L(U,X)}.
\end{align*}
Thus, the condition $B\in L(U,X_{-1+\alpha})$ is equivalent to $(\omega I-A)^{-1+\alpha}B \in L(U,X)$.
With this in mind, we have
\begin{eqnarray}
T_{-1}(t)B &=& T_{-1}(t)(\omega I-A)^{1-\alpha}(\omega I-A)^{-1+\alpha}B.
\label{eq:Analytic-systems-1}
\end{eqnarray}
Due to \cite[Theorem 2.6.13, p. 74]{Paz83}, on $X_{1-\alpha} = D((\omega I-A)^{1-\alpha})$ it holds that
\[
T_{-1}(t)(\omega I-A)^{1-\alpha} = (\omega I-A)^{1-\alpha}T_{-1}(t).
\]
Now take any $f \in L^q_{\loc}(\R_+,U)$ with $q>\frac{1}{\alpha}$. Representing 
\begin{align*}
\int_0^t T_{-1}(t-s) Bf(s) ds 
= \int_0^t (\omega I-A)^{1-\alpha} T_{-1}(t-s) (\omega I-A)^{\alpha-1} Bf(s) ds
\end{align*}
and applying item (i) of this proposition and in particular the estimate \eqref{eq:Convolution-boundedness} with $d:=1-\alpha$, $p:=q$, and with $g:=(\omega I-A)^{-1+\alpha}B f$ we see that $B$ is zero-class $q$-admissible for $q \in (\frac{1}{\alpha},+\infty)$.

\textbf{(iii).} It holds that
$\|(\omega I-A)^{d} B\|_{L(U,X_{-1+\alpha-d})} = \|(\omega I-A)^{-1+\alpha}B\|_{L(U,X)}$, 
and item (i) implies the claim. 
\end{proof}

%
%
%
%
%
%
%
 %

\begin{proposition}
\label{prop:Continuity-analytic-systems} 
Assume that $A$ generates an analytic semigroup $T$ and $B\in L(U,X_{-1+\alpha})$ for some $\alpha \in (0,1)$. 

For any $\omega>\omega_0(T)$, any $d \in[0,\alpha)$, and any $\kappa>\omega_0(T)$ there is $R>0$ such that for any $g \in L^\infty_{\loc}(\R_+,X)$ the map
	\begin{align}
	\label{eq:Convolution-map-analytic}
	t\mapsto \int_0^t (\omega I-A)^d  T(t-s) Bg(s) ds
	\end{align}
is continuous in $X$-norm, and the following holds:
\begin{align}
\int_0^t \big\|(\omega I-A)^d  T(t-s) Bg(s)\big\|_X ds  \leq R t^{\alpha-d} e^{\kappa t} \| g\|_{L^\infty([0,t],X)}.
\label{eq:Generalized-admissibility-estimate-analytic}
\end{align}
\end{proposition}

\begin{proof}
For $d<\alpha$ and $g \in L^\infty_{\loc}(\R_+,U)$, consider the map
\begin{eqnarray*}
s \mapsto (\omega I-A)^d T_{-1}(t-s)Bg(s) &=& (\omega I-A)^{1-\alpha+d}T_{-1}(t-s)(\omega I-A)^{-1+\alpha}Bg(s).
\end{eqnarray*}
By item (i) of Proposition~\ref{prop:ISS-analytic-systems}, this map is Bochner integrable and in view of \eqref{eq:Convolution-analytic-for-Linfty-norm} with $1-\alpha+d$ instead of $d$ and $(\omega I-A)^{-1+\alpha}Bg$ instead of $g$, we see that the map 
	\eqref{eq:Convolution-map-analytic} is continuous and \eqref{eq:Generalized-admissibility-estimate-analytic} holds. 
\end{proof}

\subsection{Semilinear analytic systems and their mild solutions}
\label{sec:Semilinear analytic systems and their mild solutions}

Consider again the system \eqref{eq:SEE+admissible} with $B_2=\id$ that we restate next:
\begin{subequations}
\label{eq:SEE+admissible-analytic} 
\begin{eqnarray}
\dot{x}(t) & = & Ax(t) + f(x(t),u(t)) + Bu(t),\quad t>0,  \label{eq:SEE+admissible-analytic-1}\\
x(0)  &=&  x_0, \label{eq:SEE+admissible-analytic-2}
\end{eqnarray}
\end{subequations}
In Section~\ref{sec:Semilinear boundary control systems}, we have assumed that $f$ is a well-defined map from $X \tm U$ to $X$. 
Although it sounds natural, it is, in fact, a quite restrictive assumption, as already basic nonlinearities, such as \amc{pointwise} polynomial maps, do not satisfy it. 
Indeed, if $f(x) = x^2$, where $x \in X:=L^2(0,1)$, then $f$ maps $X$ to the space $L^1(0,1)$.
However, as $A$ generates an analytic semigroup, the requirements on $f$ can be considerably relaxed. 
Namely, we assume in this section that there is $\alpha \in [0,1]$ such that  $f$ is a well-defined map from $X_\alpha \tm U $ to $X$.

We note that systems \eqref{eq:SEE+admissible-analytic} without inputs ($u=0$) have been analyzed several decades ago, see the classical monographs \cite{Hen81,Paz83}. The main difference to these works is the presence of unbounded input operators.

Next, we define mild solutions of \eqref{eq:SEE+admissible-analytic}. Note that the nonlinearity $f$ is defined on $X_\alpha \tm U$, and thus we must require that the mild solution lies in $X_\alpha$ for all positive times. We cannot expect such a nice behavior for general semigroups, but thanks to the smoothing effect of analytic semigroups, this is what we can expect in the analytic case.
\begin{definition}
\label{def:Mild-solution-analytic}
\index{solution!mild}
Let $\tau>0$ and $\alpha\in[0,1]$ be given. 
A function $x \in C([0,\tau], X)$ is called a \emph{mild solution of \eqref{eq:SEE+admissible-analytic} on $[0,\tau]$} corresponding to certain $x_0\in X$ and $u \in L^1_{\loc}(\R_+,U)$, if $x(s) \in X_\alpha$ for $s\in(0,\tau]$, and $x$ solves the integral equation
\begin{align}
\label{eq:SEE+admissible_Integral_Form-analytic}
x(t)=T(t) x_0 + \int_0^t T(t-s) f\big(x(s),u(s)\big)ds + \int_0^t T_{-1}(t-s) Bu(s)ds. 
\end{align}

We say that $x:\R_+\to X$ is a \emph{mild solution of \eqref{eq:SEE+admissible-analytic} on $\R_+$} corresponding to 
certain $x_0\in X$ and $u \in L^1_{\loc}(\R_+,U)$, if \amc{its restriction to $[0,\tau]$} is a mild solution of \eqref{eq:SEE+admissible-analytic} (with $x_0, u$) on $[0,\tau]$ for all $\tau>0$.
\end{definition}

\begin{remark}
Note that if $\alpha = 0$, then $X_\alpha = X_0 = X$, and the concept of a mild solution introduced for general and analytic semigroups coincide.
\end{remark}

\begin{ass}
\label{ass:Regularity-f-analytic} 
Let the following hold:
\begin{enumerate}[label=(\roman*)]
	\item $\alpha \in (0,1)$.
	\item $B\in L(U,X_{-1+\alpha+\varepsilon})$ for sufficiently small $\varepsilon>0$.
	
	\item $f \in C(X_\alpha \tm U, X)$, and $f$ is Lipschitz continuous in the first argument in the following sense: for each $r>0$ there is $L=L(r)>0$ such that for each $x_1,x_2 \in B_{r,X_\alpha}$ and all $u \in B_{r,U}$ it holds that 
\begin{eqnarray}
\|f(x_1,u) - f(x_2,u)\|_X \leq L\|x_1-x_2\|_{X_\alpha}.
\label{eq:Lipschitz-in-x-power-alpha}
\end{eqnarray}
	\item For all $u\in L^\infty(\R_+,U)$ and any $x\in C(\R_+,X)$ with $x((0,+\infty)) \subset X_\alpha$, the map $s\mapsto f\big(x(s),u(s)\big)$ is in $L^p_{\loc}(\R_+,X)$ with a certain $p>\frac{1}{1-\alpha}$.
	\item There is $\sigma\in\Kinf$ such that 
\[
\|f(0,u)\|_X \leq \sigma(\|u\|_U) + c,\quad u\in U.
\]
\end{enumerate}
\end{ass}

%
%

\subsection{Local existence and uniqueness}

By Proposition~\ref{prop:ISS-analytic-systems}, the condition $B\in L(U,X_{-1+\alpha+\varepsilon})$ with $\alpha,\varepsilon>0$ implies that $B$ is zero-class $q$-admissible for any $q\in(\frac{1}{\alpha+\varepsilon},+\infty]$. This in turn implies that for such $q$ the map $t\mapsto \int_0^t T_{-1}(t-s) Bu(s)ds$ is continuous for any $u \in L^q(\R_+,U)$, by \cite[Proposition 2.3]{Wei89b}.

By Assumption~\ref{ass:Regularity-f-analytic}(iv), we see that for any $u\in L^\infty(\R_+,U)$ the map 
\[
t \mapsto \int_0^t T(t-s)f\big(x(s),u(s)\big)ds
\]
is well-defined and continuous.

Hence, if $x\in C(\R_+,X)$ with $x((0,+\infty)) \subset X_\alpha$, then for any  $u\in L^\infty(\R_+,U)$ the right-hand side of 
\eqref{eq:SEE+admissible_Integral_Form-analytic} is a continuous function of time.

Our next result is the local existence and uniqueness theorem for analytic systems with initial states in $X_\alpha$ and the inputs in $\Uc:=L^\infty_{\loc}(\R_+,U)$. Recall the notation $\Uc_\Sc$ from \eqref{eq:Inputs-constrained}.
%

\begin{theorem}[Picard-Lindel\"of theorem for analytic systems]
\label{PicardCauchy-analytic}
Let Assumption~\ref{ass:Regularity-f-analytic} hold.
Assume that $T$ is an analytic semigroup, satisfying for certain $M \geq 1$, $\lambda>0$ the estimate
\begin{eqnarray*}
\|T(t)\| \leq Me^{\lambda t},\quad t\geq 0.
\end{eqnarray*} 
For any compact set $Q \subset X_\alpha$, any $r>0$, any bounded set $\Sc \subset U$, and any $\delta>0$, there is a time $t_1 = t_1(Q,r,\Sc,\delta)>0$, such that for any $w \in Q$, any $x_0 \in W:=B_{r,X_\alpha}(w)$, and any $u\in\Uc_{\Sc}$ there is a unique mild solution of \eqref{eq:SEE+admissible} on $[0,t_1]$, and it lies in the ball $B_{Mr+\delta,X_\alpha}(w)$. 
%
\end{theorem}

\begin{proof}
First, we show the claim for the case if $Q=\{w\}$ is a single point in $X_\alpha$.

\textbf{(i).} Take any $\omega>\omega_0(T)$, and consider the corresponding space $X_\alpha$. 
Pick any $r>0$ and any $C>0$ such that $W := B_{r,X_\alpha}(w) \subset B_{C,X_\alpha}$, and $\Uc_\Sc \subset B_{C,\Uc}$. Pick any $u \in \Uc_\Sc$. Take also any $\delta>0$, and consider the following sets (depending on a parameter $t>0$):
\begin{eqnarray*}
Y_{t}:= \big\{y \in C([0,t],X): \|y(s) - (\omega I-A)^\alpha w\|_X\leq Mr + \delta \ \ \forall s \in [0,t]\big\},
\end{eqnarray*}
endowed with the metric $\rho_{t}(y_1,y_2):=\sup_{s \in [0,t]} \|y_1(s)-y_2(s)\|_X$, which makes $Y_t$ complete metric spaces for all $t>0$.
%

\textbf{(ii).} Pick any $x_0 \in W$. We are going to prove that $C([0,t],X)$ is invariant under the operator
$\Phi_u$, defined for any $y \in Y_{t}$ and all $\tau \in [0,t]$ by
\begin{align}
\label{eq:FP-Operator}
\Phi_u(y)(\tau) 
&= (\omega I-A)^\alpha T(\tau)x_0  + \int_0^\tau (\omega I-A)^\alpha T_{-1}(\tau-s) Bu(s)ds  \nonumber\\
&\qquad\qquad\qquad + \int_0^\tau (\omega I-A)^\alpha T(\tau-s)f\big((\omega I-A)^{-\alpha} y(s),u(s)\big)ds.
\end{align}

Since $y\in C([0,t],X)$, the map $s \mapsto (\omega I-A)^{-\alpha} y(s)$ is in $C([0,t],X_\alpha)$, as for any $s_1,s_2 \in[0,\tau]$ we have that 
\begin{align*}
\big\|(\omega I-A)^{-\alpha} y(s_1) - (\omega I-A)^{-\alpha} y(s_2)\big\|_{X_\alpha} = \|y(s_1) - y(s_2)\|_X.
\end{align*}

By Assumption~\ref{ass:Regularity-f-analytic}, the map $s\mapsto f\big((\omega I-A)^{-\alpha} y(s),u(s)\big)$ is in $L^p_{\loc}(\R_+,X)$, with a certain $p>\frac{1}{1-\alpha}$. 
Proposition~\ref{prop:ISS-analytic-systems} ensures, that the map
\[
\tau\mapsto \int_0^\tau (\omega I-A)^\alpha T(\tau-s)f\big((\omega I-A)^{-\alpha} y(s),u(s)\big)ds
\]
is continuous.

Since $B\in L(U,X_{-1+\alpha+\varepsilon})$, Proposition~\ref{prop:ISS-analytic-systems}(ii) implies that
\[
\tau \mapsto \int_0^\tau (\omega I-A)^\alpha  T_{-1}(\tau-s)  Bu(s)ds
\]
belongs to $C([0,\tau],X)$.
%
%
%

Overall, the function $\Phi_u(y)$ is continuous, and thus $\Phi_u$ maps $C([0,t],X)$ to $C([0,t],X)$.

\textbf{(iii).} 
Now we prove that for small enough $t$ the spaces $Y_{t}$  are invariant under the operator $\Phi_u$.

Fix any $t>0$ and pick any $y \in Y_{t}$.
As $x_0 \in W=B_{r,X_\alpha}(w)$, there is $a \in X_\alpha$: $\|a\|_{X_\alpha} <r$ such that $x_0=w+a$.
 
Then for any $\tau<t$, we obtain that
\begin{align*}
\|&\Phi_{t}(y)(\tau)- (\omega I-A)^\alpha w\|_X \nonumber\\
&\leq   \Big\|(\omega I-A)^\alpha T(\tau)x_0 - (\omega I-A)^\alpha w\Big\|_X + \Big\|\int_0^\tau (\omega I-A)^\alpha T_{-1}(\tau-s) Bu(s)ds\Big\|_X \nonumber\\
&\qquad\qquad\qquad + \int_0^\tau \big\|(\omega I-A)^\alpha T(\tau-s)\big\| \big\|f((\omega I-A)^{-\alpha}y(s),u(s))\big\|_Xds. \nonumber
\end{align*}
We substitute $x_0:=w+a$ into the first term on the right-hand side of the above inequality. 
The last term we estimate using \eqref{eq:Bound-on-A^alphaT(t)}.
To estimate the second term, we use that $(\omega I-A)^{\alpha}  B \in L(U,X_{-1+\varepsilon})$. 
\amc{By} Proposition~\ref{prop:ISS-analytic-systems}, $(\omega I-A)^{\alpha} B$ is zero-class $\infty$-admissible, and thus there is an increasing continuous function $t \mapsto h_t$ satisfying $h_0=0$, such that:
\begin{align}
\label{eq:technical-estimates-input-operator}
\big\|&\Phi_{t}(y)(\tau)- (\omega I-A)^\alpha w\big\|_X \leq \big\|(\omega I-A)^\alpha T(\tau)w - (\omega I-A)^\alpha w\big\|_X \\
			& \qquad + \|(\omega I-A)^\alpha T(\tau)a\|_X  + h_{\tau} \|u\|_{L^\infty([0,\tau],U)} \nonumber\\
			&\qquad+ \int_0^\tau \frac{C_\alpha e^{\lambda (\tau-s)}}{(\tau-s)^\alpha} \big(\|f(0,u(s))\|_X \nonumber\\
			&\qquad\qquad\qquad\qquad\qquad+ \big\|f((\omega I-A)^{-\alpha}y(s),u(s))-f(0,u(s))\big\|_X\big)ds. \nonumber
\end{align}

To estimate the latter expression, note that 
\begin{itemize}
	\item $\|(\omega I-A)^\alpha a\|_X = \|a\|_{X_\alpha} < r$. 
	\item For all $s\in[0,t]$ we have 
\begin{align*}
\|(\omega I-A)^{-\alpha}y(s)-0\|_{X_\alpha} = \|y(s)\|_X &\leq \|(\omega I-A)^\alpha w\|_X+Mr+\delta\\
&\leq M(\|w\|_{X_\alpha}+r) +\delta\leq MC+\delta :=K.
\end{align*}
	\item In view of Assumption~\ref{ass:Regularity-f-analytic}, it holds that 
\[
\|f(0,u(s))\|_X \leq \sigma(\|u(s)\|_U) + c,\quad \text{ for a.e. } s \in [0,t].
\]
	\item $h$ is a monotonically increasing continuous function.
\end{itemize}
As $M\geq 1$, it holds that $K>C$, and Lipschitz continuity of $f$ on bounded balls ensures that there is $L(K)>0$, such that for all $\tau\in[0,t]$ 			
\begin{align*}
\|&\Phi_{t}(y)(\tau)-(\omega I-A)^\alpha w\|_X \\			
			&\leq \|(\omega I-A)^\alpha T(\tau)w - (\omega I-A)^\alpha w\|_X + \|T(\tau)(\omega I-A)^\alpha a\|_X 
			+ h_{\tau} \|u\|_{L^\infty([0,\tau],U)}\\
			&\quad+ \int_0^\tau \frac{C_\alpha}{(\tau-s)^\alpha} e^{\lambda (\tau-s)} \big(\sigma(\|u(s)\|_U) + c + L(K)\|(\omega I-A)^{-\alpha}y(s)\|_{X_\alpha}\big)ds\\
			&\leq \sup_{\tau\in[0,t]}\|T(\tau)(\omega I-A)^\alpha w - (\omega I-A)^\alpha w\|_X + Me^{\lambda t}r + h_{t} \|u\|_{L^\infty([0,t],U)}\\
			&\quad+ C_\alpha e^{\lambda t} \big(\sigma(C) + c + L(K)K\big) \int_0^t \frac{1}{s^\alpha} ds.
\end{align*}
Since $T$ is a strongly continuous semigroup, and $h_t\to 0$ as $t\to +0$, from this estimate, it is clear that  there exists $t_1$, such that 
\[
\|\Phi_u(y)(t)-w\|_X \leq Mr +\delta,\quad \text{ for all } t \in [0,t_1].
\]
This means, that $Y_{t}$ is invariant with respect to $\Phi_u$ for all $t \in (0,t_1]$, and $t_1$ does not depend on the choice of $x_0 \in W$.

\textbf{(iv).} 
Now pick any $t>0$, $\tau \in [0, t]$, and any $y_1, y_2 \in Y_{t}$. Then it holds that
\begin{align*}
\|\Phi_u(y_1)&(\tau) - \Phi_u(y_2)(\tau)\|_X \\
&\leq \int_0^\tau \|(\omega I-A)^\alpha T(\tau-s)\|\\
&\qquad\qquad \cdot\big\|f((\omega I-A)^{-\alpha}y_1(s),u(s)) - f((\omega I-A)^{-\alpha}y_2(s),u(s))\big\|_Xds \\
&\leq 
\int_0^t L(K)\frac{C_\alpha}{(\tau-s)^\alpha} e^{\lambda (\tau-s)} \|y_1(s)-y_2(s)\|_X ds \\
&\leq L(K) C_\alpha e^{\lambda t} \int_0^t s^{-\alpha} ds \rho_t(y_1,y_2) \\
&\leq L(K) C_\alpha e^{\lambda t} \frac{t^{1-\alpha}}{1-\alpha} \rho_t(y_1,y_2) \\
&\leq \frac{1}{2} \rho_t(y_1,y_2),
\end{align*}
for $t \leq t_2$, where $t_2>0$ is a small enough real number that does not depend on the choice of $x_0 \in W$.

%

\ifnothabil	\sidenote{\mir{Reference for Full book:\quad	Theorem~\ref{thm:Banach fixed point theorem}.}}\fi

According to Banach fixed point theorem, there exists a unique $y \in Y_t$ that is a fixed point of $\Phi_u$, that is 
\begin{align}
\label{eq:FP-Operator-solution}
y(\tau) 
&= (\omega I-A)^\alpha T(\tau)x_0  + \int_0^\tau (\omega I-A)^\alpha T_{-1}(\tau-s) Bu(s)ds  \nonumber\\
&\qquad\qquad\qquad + \int_0^\tau (\omega I-A)^\alpha T(\tau-s)f\big((\omega I-A)^{-\alpha} y(s),u(s)\big)ds.
\end{align}
on $[0,\min\{t_1,t_2\}]$. 

As $(\omega I-A)^{\alpha}$ is invertible with a bounded inverse, $y$ solves \eqref{eq:FP-Operator-solution} if and only if 
$y$ solves
\begin{align}
\label{eq:FP-Operator-solution-II}
(\omega I-A)^{-\alpha} y(\tau) 
&= T(\tau)x_0  + \int_0^\tau T_{-1}(\tau-s) Bu(s)ds  \nonumber\\
&\qquad\qquad\qquad + \int_0^\tau T(\tau-s)f\big((\omega I-A)^{-\alpha} y(s),u(s)\big)ds.
\end{align}
As $y \in C([0, \min\{t_1,t_2\}],X)$, the map $x:=(\omega I-A)^{-\alpha} y$ is in $C([0, \min\{t_1,t_2\}],X_\alpha)$, and is the unique mild solution of \eqref{eq:SEE+admissible-analytic}. 

\textbf{(v). General compact $Q$.} 
 Similar to the corresponding part of the proof of Theorem~\ref{PicardCauchy}.
\end{proof}

\begin{remark}
\label{rem:Existence for initial states in X} 
For systems without inputs, Theorem~\ref{PicardCauchy-analytic} was shown (in a somewhat different formulation without bounds on the growth of the solution) in \cite[Theorem 3.1]{Paz83}. 
We have proved our local existence result for initial conditions that are in $X_\alpha$. To ensure local existence and uniqueness 
for the initial states outside of $X_\alpha$, stronger requirements on $f$ have to be imposed, see \cite[Theorems 7.1.5, 7.1.6]{Lun12}.
\end{remark}

Introducing the concepts of maximal solutions and of well-posedness and arguing similar to Sections~\ref{sec:Local existence and uniqueness}, \ref{sec:Global-Well-posedness}, we obtain the following well-posedness theorem.
\begin{theorem}
\label{thm:Well-posedness-analytic}
Let $A$ generate an analytic semigroup, Assumption~\ref{ass:Regularity-f-analytic} hold, and let $\Uc:=L^\infty(\R_+,U)$. 
Then:
\begin{enumerate}[label=(\roman*)]
	\item For each $x\in X_\alpha$ and each $u \in\Uc$, there is a unique maximal solution of \eqref{eq:SEE+admissible-analytic}, defined over the certain maximal time-interval $[0,t_m(x,u))$. We denote this solution as $\phi(\cdot,x,u)$.
	\item The triple $\Sigma:=(X_\alpha,\Uc,\phi)$ is a well-defined control system in the sense of Definition~\ref{Steurungssystem}.
	\item $\Sigma$ satisfies the BIC property, that is if for a certain $x \in X_\alpha$ and $u\in\Uc$ we have $t_m(x,u)<\infty$, then 
	 $\|\phi(t,x,u)\|_{X_\alpha} \to \infty$ as $t \to t_m(x,u)-0$.
\end{enumerate}
\end{theorem}

\subsection{Global existence}


%

Motivated by \cite[Section 6.3, Theorem 3.3]{Paz83}, we have the following result guaranteeing the forward completeness and BRS property for semilinear analytic systems.
\begin{theorem}
\label{thm:Global-Existence-Nonlinear-Eq} 
Let $A$ generate an analytic semigroup, Assumption~\ref{ass:Regularity-f-analytic} hold, and let $\Uc:=L^\infty(\R_+,U)$. 
Assume further that there are $L,c>0$ and $\sigma\in\Kinf$ such that 
\begin{eqnarray}
\|f(x,u)\|_X \leq L\|x\|_{X_\alpha} + \sigma(\|u\|_U) + c,\quad x\in X_\alpha,\quad u\in U.
\label{eq:f-linearly-bounded}
\end{eqnarray} 
Then  $\Sigma:=(X_\alpha,\Uc,\phi)$ is a forward complete control system.
\end{theorem}

\begin{proof}
Take any positive $\omega >\omega_0(T)$ and define $X_\alpha$ as in \eqref{eq:X-power-alpha}.

We argue by a contradiction.
Let $\Sigma$ be not forward complete. Then there are $(x_0,u)\in X_\alpha \tm \Uc$ such that $t_m(x_0,u)<\infty$. By Theorem~\ref{thm:Well-posedness-analytic}, we have that $\|\phi(t,x_0,u)\|_{X_\alpha} \to \infty$ as $t \to t_m(x_0,u)-0$.

For $t<t_m(x_0,u)$ denote $x(t):=\phi(t,x_0,u)$. 
As $x(\cdot) \subset X_\alpha$, we can apply $(\omega I-A)^\alpha$ along the trajectory $x(\cdot)$ to obtain
\begin{align*}
(\omega I-A)^\alpha x(t) &= (\omega I-A)^\alpha T(t)x_0 + \int_0^\tau (\omega I-A)^\alpha T_{-1}(\tau-s) Bu(s)ds  \\
			&\qquad	+ \int_{0}^t (\omega I-A)^\alpha T(t-s)f(x(s),u(s))ds.
\end{align*}
We obtain 
\begin{align*}
\|x(t)\|_\alpha 
&= \|(\omega I-A)^\alpha x(t)\|_X  \\
&\leq \|(\omega I-A)^\alpha T(t)x_0\|_X + \Big\|\int_0^\tau (\omega I-A)^\alpha T_{-1}(\tau-s) Bu(s)ds\Big\|_X  \\
			&\qquad\qquad	+ \int_{0}^t \big\|(\omega I-A)^\alpha T(t-s)\big\| \big\|f(x(s),u(s))\big\|_Xds.
\end{align*}
We now estimate the second term as in \eqref{eq:technical-estimates-input-operator}, where $h$ is a continuous increasing function with $h_0=0$. The last term we estimate using \eqref{eq:Bound-on-A^alphaT(t)}. Overall:
\begin{align*}
\|x(t)\|_\alpha 
&\leq Me^{\omega t} \|(\omega I-A)^\alpha x_0\|_X + h_{\tau} \|u\|_{L^\infty([0,t],U)}\\
			&\qquad\qquad	+ \int_{0}^t \frac{C_\alpha}{(t-s)^\alpha}e^{\omega (t-s)} 
\Big(L\|x(s)\|_{X_\alpha} + \sigma(\|u(s)\|_U) + c \Big)ds.
\end{align*}
Defining $z(t):=x(t)e^{-\omega t}$, we obtain from the previous estimate that
\begin{align*}
\|z(t)\|_\alpha 
&\leq M\|(\omega I-A)^\alpha x_0\|_X + \int_{0}^t \frac{C_\alpha}{s^\alpha}ds  \Big(\sigma(\|u\|_{L^\infty(\R_+,U)}) + c \Big)
+ h_{\tau} \|u\|_{L^\infty(\R_+,U)}\\
&\qquad + \int_{0}^t \frac{L C_\alpha}{(t-s)^\alpha}\|z(s)\|_{X_\alpha}ds.  
\end{align*}
An analytic version of Gronwall inequality \cite[p. 6]{Hen81} 
shows that $z$, and hence $x$, is uniformly bounded on $[0,t_m(x_0,u))$, and BIC property (Theorem~\ref{thm:Well-posedness-analytic}(iii)) shows that $t_m(x_0,u)$ is not the finite maximal existence time. A contradiction.
\end{proof}


\subsection{Example: well-posedness of a Burgers' equation with a distributed input}

We consider the following semilinear reaction-diffusion equation of Burgers' type on a domain $[0,\pi]$, with distributed input $u$, boundary input $d$ at $z=0$, and homogeneous Dirichlet boundary condition at $\pi$.
\begin{subequations}
\label{eq:Nonlinear-Burgers-Equation}
\begin{align}
x_t &= x_{zz} - xx_z  + f(z,x(z,t))+ u(z,t),\quad z\in (0,\pi),\quad t>0,\\
x&(0,t)=d(t),\quad t>0,\\
x&(\pi,t)=0.
\end{align}
\end{subequations}

Here $f: [0,\pi]\tm \R \to\R$ is  measurable  in  $z$,  locally  Lipschitz  continuous  in  $x$ uniformly  in  $z$, and 
\begin{eqnarray}
|f(z,y)| \leq h(z)g(|y|),\quad \text{for a.e. } z \in[0,\pi], \text{ and all } y \in\R,
\label{eq:f-bound}
\end{eqnarray}
where $h\in L^2(0,\pi)$, and $g$ is continuous, increasing, and both $h$ and $g$ are positive.

This system with $u=0$ and $d=0$ was investigated in \cite[p. 57]{Hen81}. Here we give a detailed analysis of this system with distributed and boundary inputs.

We denote $X:=L^2(0,\pi)$. The operator $A:= \frac{d^2}{dz^2}$ with the domain $D(A) = H^2(0,\pi)\cap H^1_0(0,\pi)$ generates an analytic semigroup on $X$.

We assume that the distributed input $u$ belongs to the space $\Uc=L^\infty(\R_+,U)$, with $U:=L^2(0,\pi)$, and the boundary input $d$ belongs to $\Dc:=L^\infty(\R_+,\R)$.

The system \eqref{eq:Nonlinear-Burgers-Equation} can be reformulated as a semilinear evolution equation
\begin{eqnarray}
x_t = Ax + F(x)+u + Bd,
\label{eq:Nonlinear-Burgers-Equation-Evolution-Eq}
\end{eqnarray}
where we slightly abuse the notation and use $x$ as an argument of the evolution equation.

The condition (ii) in Assumption~\ref{ass:Regularity-f-analytic} characterizing the admissibility properties of the boundary input operator $B$ holds in view of \cite[Example 2.16]{Sch20}.

  
The space $X_{\frac{1}{2}}$ corresponding to the operator $A$, is given by (see \cite[Proposition 3.6.1]{TuW09}) 
\[
X_{\frac{1}{2}} = H^1_0(0,\pi),
\]
which is a Banach space with the norm
\[
\|x\|_{\frac{1}{2}}:=\Big|\int_0^\pi |x'(z)|^2 dz\Big|^{\frac{1}{2}}, \quad x \in X_{\frac{1}{2}}.
\]

The nonlinearity $F: X_{\frac{1}{2}} \to X$ in \eqref{eq:Nonlinear-Burgers-Equation-Evolution-Eq} is given by 
\[
F(x)(z) = - x(z)x'(z) + f(z,x(z)).
\]

\begin{proposition}
\label{prop:Burgers-well-posedness} 
For each $x_0 \in X_{\frac{1}{2}}$, each $u \in \Uc=L^\infty_{\loc}(\R_+,U)$, and each boundary input $d \in \Dc=L^\infty_{\loc}(\R_+,\R)$ the system \eqref{eq:Nonlinear-Burgers-Equation} possesses a unique maximal mild solution $\phi(\cdot,x_0,(u,d))$.
The system $\Sigma = (X_{\frac{1}{2}},\Uc\tm\Dc,\phi)$ is a control system satisfying the BIC property.
\end{proposition}

\begin{proof}
We proceed in 3 steps:

\textbf{Step 1: $F$ maps bounded sets of $ X_{\frac{1}{2}}$ to bounded sets of $X$.} 
Since the elements of $X_{\frac{1}{2}}=H^1_0(0,\pi)$ are absolutely continuous functions, using the Cauchy-Schwarz inequality, we obtain that for any $x \in X_{\frac{1}{2}}$ it holds that 
\begin{eqnarray}
\sup_{z\in(0,\pi)}|x(z)| 
&=& \sup_{z\in(0,\pi)}\Big|\int_0^z x'(z) dz\Big| 
\leq \sup_{z\in(0,\pi)}\int_0^z |x'(z)| dz 
 =  \int_0^\pi |x'(z) |dz  \nonumber\\
&\leq & \Big|\int_0^\pi 1 dz\Big|^{\frac{1}{2}}  \Big|\int_0^\pi |x'(z)|^2 dz\Big|^{\frac{1}{2}} 
= \sqrt{\pi} \|x\|_{\frac{1}{2}}.
\label{eq:sup-norm-estimated-by-Sobolev-norm}
\end{eqnarray}

For any $x \in X_{\frac{1}{2}}$ consider
\begin{eqnarray*}
\|F(x)\|_X^2
= \int_0^\pi |F(x)(z)|^2 dz
&=& \int_0^\pi |x(z)x'(z) + f(z,x(z))|^2 dz\\
&\leq &\int_0^\pi 2|x(z)x'(z)|^2 + 2|f(z,x(z))|^2 dz.
\end{eqnarray*}
Using \eqref{eq:sup-norm-estimated-by-Sobolev-norm} and \eqref{eq:f-bound}, we continue the estimates as follows:
\begin{eqnarray*}
\|F(x)\|_X^2
&\leq &\int_0^\pi 2|x'(z)|^2 \pi \|x\|^2_{\frac{1}{2}} + 2 |h(z)|^2  |g(|x(z)|)|^2 dz\\
&\leq & 2\pi \|x\|^4_{\frac{1}{2}}  + 2 \int_0^\pi |h(z)|^2 |g(\sqrt{\pi} \|x\|_{\frac{1}{2}})|^2 dz\\
&= & 2\pi \|x\|^4_{\frac{1}{2}}  + 2 \|h\|^2_X |g(\sqrt{\pi} \|x\|_{\frac{1}{2}})|^2.
\end{eqnarray*}
 Taking the square root and using that $\sqrt{a+b}\leq \sqrt{a} + \sqrt{b}$, for all $a,b\geq 0$, we finally obtain
\begin{eqnarray}
\|F(x)\|_X
&\le & \sqrt{2\pi} \|x\|^2_{\frac{1}{2}}  + \sqrt{2} \|h\|_X |g(\sqrt{\pi} \|x\|_{\frac{1}{2}})|.
\label{eq:Estimate-on-F}
\end{eqnarray}
This shows that $F$ is well-defined as a map from $X_{\frac{1}{2}}$ to $X$ and $F$ maps bounded sets of $X_{\frac{1}{2}}$ to bounded sets of $X$.

\textbf{Step 2: $F$ is Lipschitz continuous in $x$.} 
For $x_0 \in X_{\frac{1}{2}}$, there is a neighborhood $V$ of a compact set 
$
\{(z,x_0(z)):z\in[0,\pi]\}
$ 
in $[0,\pi]\tm \R$ and positive constants $L,\theta$ so that for
$(z,x_1) \in V$, $(z,x_2) \in V$ it holds that 
\[
|f(z,x_1) - f(z,x_2)| \leq L |x_1-x_2|.
\]
Thus, there is a neighborhood $U$ of $x_0$ in $X_{\frac{1}{2}}$
such that  $x\in U$ implies that $(z,x(z))\in V$ for a.e. $z\in[0,\pi]$ and for $x_1,x_2 \in U$ it holds that
\begin{align*}
\big\|f(\cdot,x_1(\cdot)) - f(\cdot,x_2(\cdot))\big\|_{X}^2
&= \int_0^\pi\big|f(z,x_1(z)) - f(z,x_2(z))\big|^2dz \\
&\leq L^2  \int_0^\pi|x_1(z)-x_2(z)|^2 dz,
\end{align*}
and using \eqref{eq:sup-norm-estimated-by-Sobolev-norm} we have that
\begin{eqnarray*}
\big\|f(\cdot,x_1(\cdot)) - f(\cdot,x_2(\cdot))\big\|_{X}^2
&\leq& L^2 \pi^2 \|x_1-x_2\|^2_{\frac{1}{2}}.
\end{eqnarray*}
Taking the square root, we have that  
\begin{eqnarray}
\big\|f(\cdot,x_1(\cdot)) - f(\cdot,x_2(\cdot))\big\|_{X}
\leq \pi L \|x_1-x_2\|_{\frac{1}{2}}.
\label{eq:Lipschitz estimate-for-f}
\end{eqnarray}
Finally, for any $x_1,x_2\in X_{\frac{1}{2}}$ it holds that 
\begin{eqnarray*}
\|x_1x_1' - x_2x_2'\|_X &\leq& \|x_1(x_1'-x_2')\|_X + \|(x_1-x_2)x_2'\|_X,
\end{eqnarray*}
and again using \eqref{eq:sup-norm-estimated-by-Sobolev-norm}, we proceed to 
\begin{eqnarray}
\|x_1x_1' - x_2x_2'\|_X 
&\leq& \sqrt{\pi} \|x_1\|_{\frac{1}{2}}  \|x_1-x_2\|_{\frac{1}{2}} + \sqrt{\pi} \|x_1-x_2\|_{\frac{1}{2}} \|x_2\|_{\frac{1}{2}}\nonumber\\
&=& \sqrt{\pi} ( \|x_1\|_{\frac{1}{2}}  + \|x_2\|_{\frac{1}{2}})\|x_1-x_2\|_{\frac{1}{2}}.
\label{eq:Lipschitz estimate-for-cross-term}
\end{eqnarray}
Combining \eqref{eq:Lipschitz estimate-for-f} and \eqref{eq:Lipschitz estimate-for-cross-term}, we obtain the required Lipschitz property for the function $F$.

\textbf{Step 3: Application of general well-posedness theorems.} Finally, Theorem~\ref{PicardCauchy-analytic} shows that the system \eqref{eq:Nonlinear-Burgers-Equation} possesses a unique mild solution for each $x_0 \in X_{\frac{1}{2}}$, each $u \in \Uc=L^\infty_{\loc}(\R_+,U)$, and each boundary input $d \in \Dc=L^\infty_{\loc}(\R_+,\R)$. Theorem~\ref{thm:Well-posedness-analytic} shows that $\Sigma$ is a control system satisfying the BIC property.
\end{proof}

\section{Boundary control systems}
\label{sec:Boundary_control_systems}

Control systems governed by partial differential equations are defined by PDEs describing the dynamics inside of the spatial domain and boundary conditions, describing the dynamics of the system at the boundary of the domain. Such systems look (at first glance) quite differently from the evolution equations in Banach spaces, studied in Section~\ref{sec:Semilinear boundary control systems}. 
This motivated the development of a theory of abstract boundary control systems that allows for a more straightforward interpretation of PDEs in the language of semigroup theory.

\subsection{Linear boundary control systems}
\label{sec:Linear_Boundary_control_systems}

Let $X$ and $U$ be Banach spaces. Consider a system
\begin{subequations}
\label{eq:BCS}
\begin{align}
\dot{x}(t) &= {\Ah} x(t), \qquad x(0) = x_0, \label{eq:BCS-1}\\
{\Rh}x(t) &= u(t),    \label{eq:BCS-2}
\end{align}
\end{subequations}
where the \emph{formal system operator} $\Ah: D( \Ah ) \subset X \to X$ is a linear operator, the control function $u$ takes values in $U$, and the \emph{boundary operator} $\Rh : D( \Rh ) \subset X \to U$ is linear and satisfies $D(\Ah) \subset D(\Rh)$.

%
\begin{definition}
\label{def:BCS}
The system \eqref{eq:BCS} is called a~\emph{linear boundary control system (linear BCS)} if the following conditions hold:
\begin{enumerate}
\item[(i)] The operator $A : D(A) \to X$ with $D(A) = D({\Ah} ) \cap \ker({\Rh})$ defined by
    \begin{equation}
    Ax = {\Ah}x \qquad \text{for} \quad x\in D(A)
        \label{eq:BSC-ass1}
    \end{equation}
    is the infinitesimal generator of a $C_0$-semigroup $(T(t))_{t\geq0}$ on $X$.
\item[(ii)] There is an operator $R \in \mathcal{L}(U,X)$ such that for all $u \in U$ we have $Ru \in D({\Ah})$,
    ${\Ah}R \in \mathcal{L}(U,X)$ and         
        \begin{equation}
    {\Rh}Ru = u, \qquad u\in U.
        \label{eq:BSC-ass2}
    \end{equation}
\end{enumerate}
The operator $R$ in this definition is sometimes called a \emph{lifting operator}. Note that $R$ is not uniquely defined by the properties in the item (ii).
\end{definition}

%

Item (i) of Definition~\ref{def:BCS} shows that for $u\equiv 0$ the equations \eqref{eq:BCS} are well-posed.
In particular, as $A$ is the generator of a certain strongly continuous semigroup $T(\cdot)$, for any $x\in D(A)$, it holds that $T(t)x \in D(A)$ and thus $T(t)x \in \ker({\Rh})$ for all $t\geq 0$, which means that \eqref{eq:BCS-2} is satisfied.

Item (ii) of the definition implies, in particular, that the range of the operator ${\Rh}$ equals $U$, and thus the values of inputs are not restricted.

\subsection{Semilinear boundary control systems}
\label{sec:Semilinear-BCS}

Let $(\Ah,\Rh)$ be a linear BCS.
We consider $D(\Ah) \subset X$ as a linear space equipped with the graph norm
\[
\|\cdot\|_{D(\Ah)}:=\|\cdot\|_{X}+\big\|\Ah\cdot\big\|_{X}.
\]

%

Motivated by \cite{Sch20}, we consider the following class of semilinear boundary control systems.
\begin{definition}
\label{def:N-BCS}
Consider a linear BCS $(\Ah,\Rh)$.
Consider the following system 
\begin{subequations}
\label{eq:NBCS}
\begin{align}
\dot{x}(t)&= \Ah x(t)+f(x(t),w(t)), \ t>0, \label{eq:NBCS-1}\\
{\Rh}x(t) &= u(t),\ t>0,    \label{eq:NBCS-2}\\
x(0) &= x_0,    \label{eq:NBCS-3}
\end{align}
\end{subequations}
with a nonlinearity $f:X\times W\to X$, where $W$ is a Banach space. 

 
The system \eqref{eq:NBCS} we call a \emph{semilinear boundary control system (semilinear BCS)}.
\end{definition}

Following \cite{Sch20}, we define classical solutions to the semilinear BCS \eqref{eq:NBCS}.
\begin{definition}
\label{def:BCS-Classical solutions} 
Let $x_{0}\in D(\Ah)$, $\tau>0$ and $u\in C([0,\tau],U)$.  A function 
\[
x\in C([0,T],D(\Ah))\cap C^{1}([0,T],X)
\] 
is called a \emph{classical solution} to the semilinear BCS \eqref{eq:NBCS} on $[0,\tau]$ if $x(t)\in X$ for all $t>0$ and the equations \eqref{eq:NBCS} are satisfied pointwise for $t\in (0,\tau]$. 

A function $x:[0,\infty)\to X$ is called \emph{(global) classical solution} to the semilinear BCS \eqref{eq:NBCS}, if $x|_{[0,\tau]}$ is a classical solution on $[0,\tau]$ for every $\tau>0$. 

If $x\in C([0,\tau],D(\Ah))\cap C^{1}((0,\tau],X)$ and $x(t)\in X$ for all $t>0$ and the equations \eqref{eq:NBCS} are satisfied pointwise for $t\in (0,\tau]$, then we say that $x$ is a \emph{classical solution on $(0,\tau]$}.
\end{definition}

The next theorem gives a representation for the (unique) solutions of \eqref{eq:NBCS} \emph{for smooth enough inputs}.
\begin{theorem}
\label{thm:BCS-classical-solution-Representation}
Consider the boundary control system \eqref{eq:BCS} with $f \in C(X\tm W,X)$.
Let $u \in C^2([0,\tau],U)$, and $w\in C([0,\tau],W)$ for some $\tau>0$, and let $x_0\in X$ be such that $x_0 - Ru(0) \in D(A)$. Assume that the classical solution of semilinear BCS $\phi(\cdot,x_0,u)$ exists on $[0,\tau]$. Then it can be represented as
\begin{subequations}
\label{eq:BCS-solution-for-smooth-inputs}
\begin{align}
\phi(t,x_0,u) &= T(t)\big(x_0-Ru(0)\big) \nonumber\\
& \qquad + \int_0^t T(t-r)\Big(f(x(r),w(r)) + {\Ah}Ru(r)-R\dot{u}(r)\Big) dr + Ru(t) \label{eq:BCS-solution-for-smooth-inputs-1}\\
&= T(t)x_0 + \int_0^t T(t-r)\Big(f(x(r),w(r))+{\Ah}Ru(r)\Big)dr\nonumber\\
& \qquad - A\int_0^t T(t-r)Ru(r) dr \label{eq:BCS-solution-for-smooth-inputs-2}\\
&= T(t)x_0 + \int_0^t T_{-1}(t-r)\Big(f\big(x(r),w(r)\big)+({\Ah}R - A_{-1}R)u(r)\Big)dr \label{eq:BCS-solution-for-smooth-inputs-3},
\end{align}
\end{subequations}
where $A_{-1}$ and $T_{-1}$ are the extensions of the infinitesimal generator $A$ and of the semigroup $T$ to the extrapolation space $X_{-1}$.
Furthermore, $A_{-1}R \in L(U,X_{-1})$ (and thus ${\Ah}R - A_{-1}R \in L(U,X_{-1})$).
\end{theorem}

The proof of the linear case (with $f=0$) should be well-known; see, e.g., \cite[Theorem 4.4]{MiP20}, \cite[pp. 93--94]{Sch20} for the proofs of this fact, and \cite[Theorem 11.1.2]{JaZ12} for a partial result. The nonlinear result can be obtained in a similar manner. Hence we omit the proof.

An advantage of the representation formula \eqref{eq:BCS-solution-for-smooth-inputs-1} is in the boundedness of the operators $R$ and $\Ah R$ involved in the expression. Its disadvantage is that the derivative of $u$ is employed.
Still, the expression in the right-hand side of \eqref{eq:BCS-solution-for-smooth-inputs-1} makes sense for any $x \in X$ 
and for any $u \in H^1([0,\tau],U)$, and can be called a mild solution of BCS \eqref{eq:BCS}, as is done, e.g., in \cite[p. 146]{JaZ12}.

The formula \eqref{eq:BCS-solution-for-smooth-inputs-3} does not involve any derivatives of inputs, and again is given in terms of a bounded operator ${\Ah}R - A_{-1}R \in L(U,X_{-1})$.
Moreover, if we consider the expression in the right-hand side of \eqref{eq:BCS-solution-for-smooth-inputs-3} in the extrapolation spaces $X_{-1}$, then it makes sense for all $x \in X$ and all $u\in L^{1}_{\loc}(\R_+,U)$, and constitutes a mild solution of 
\begin{eqnarray}
\dot{x}(t) = Ax(t) + f(x(t),w(t)) + Bu(t),
\label{eq:semilinear equation with an admissible operator}
\end{eqnarray}
with
\begin{eqnarray}
B:={\Ah}R-A_{-1}R.
\label{eq:BCS-Input-Operator}
\end{eqnarray}
This motivates us to define the mild solutions of semilinear BCS by means of the formula \eqref{eq:BCS-solution-for-smooth-inputs-3}, as was proposed in \cite{Sch20}.

\begin{definition}
\label{def:Mild-solution-for-semilin-Boundary-Control-System}
Let $(\Ah,\Rh ,f)$ be a semilinear boundary control system with corresponding $A,R$. Let $x_{0}\in X$,  $\tau>0$, $w\in L_{\loc}^{1}([0,\tau],W)$, and $u\in L_{\loc}^{1}([0,\tau],U)$. A continuous function $x:[0,\tau]\to X$ is called \emph{mild solution} to the semilinear BCS \eqref{eq:NBCS} on $[0,\tau]$ if $x(t)\in X$ for all $t>0$ and $x$ solves 
\begin{align*}
x(t)= T(t)x_{0}+\int_{0}^{t}T_{-1}(t-s)\big(f(x(s),w(s))+Bu(s)\big)ds,
\end{align*}
for all $t\in[0,\tau]$ and where $B=\Ah B_{0}-A_{-1}B_{0}$. A function $x:\R_+ \to X$ is called a \emph{global mild solution} if $x|_{[0,\tau]}$ is a mild solution on $[0,\tau]$ for all $\tau>0$.
\end{definition}

In other words, $x$ is a mild solution of a semilinear BCS \eqref{def:N-BCS}, if $x$ is a mild solution of 
\eqref{eq:semilinear equation with an admissible operator} with $B={\Ah}R - A_{-1}R$.

Thus, semilinear boundary control systems are a special case of semilinear evolution equations studied in Section~\ref{sec:Semilinear boundary control systems}, and we can use our well-posedness theory for semilinear evolution equations to analyze semilinear BCS.

\section*{Acknowledgements} 

The author thanks the Associate Editor and anonymous Reviewers for useful and constructive comments concerning the first version of this manuscript.

\section*{Declarations}

\subsection*{Funding}

\amc{This research has been} supported by the German Research Foundation (DFG) via the grant MI 1886/2-2.

\subsection*{Conflict of interest/Competing interests}

		There is no conflict of interests.

\subsection*{Ethics approval}

		Not applicable

\subsection*{Availability of data and materials}

		Not applicable

\subsection*{Authors' contributions}

		Not applicable

\bibliographystyle{abbrv}
\bibliography{C:/Users/Andrii/Dropbox/TEX_Data/Mir_LitList_NoMir,C:/Users/Andrii/Dropbox/TEX_Data/MyPublications}


\end{document}

\appendix 
\section{Recap on fractional powers of operators}
\label{sec:Recap on fractional powers of operators}

In this section, we make a short recap of the fractional powers of sectorial operators, that we need in Section~\ref{sec:Semilinear analytic boundary control systems}. 
A somewhat more detailed account of the properties of fractional powers of sectorial operators can be found, e.g., in \cite[Section 1.4]{Hen81}
and for a detailed treatment, we refer to \cite{Haa06}.

\begin{definition}
\label{def:Gamma function} 
For $z \in\C$ with $\re z >0$ the \emph{gamma function} is defined as
\begin{eqnarray}
\Gamma(z) = \int_0^\infty t^{z-1}e^{-t}dt.
\label{eq:Gamma-function}
\end{eqnarray}
\end{definition}

\begin{definition}
\label{def:Negative-Fractional-powers-operator}
Let $A$ be the generator of an exponentially stable analytic semigroup over a Banach space $X$.
For $\alpha>0$ define
\begin{eqnarray}
(-A)^{-\alpha}:=\frac{1}{\Gamma(\alpha)}\int_0^\infty t^{\alpha-1}T(t)dt.
\label{eq:Fractional-powers-operator}
\end{eqnarray}
\end{definition}
As the semigroup $T$ is analytic, 
the map $t\to T(t)$ is differentiable on $(0,+\infty)$ (see, e.g., \cite[Theorem 4.6, p. 101]{EnN00}), and hence continuous on this interval. As $T$ is also exponentially stable, the integral 
in \eqref{eq:Fractional-powers-operator} converges in the uniform operator topology. 
The following well-known property holds:
\begin{proposition}
\label{prop:Product-formula-fractional-powers}
Let $A$ be the generator of an exponentially stable analytic semigroup over a Banach space $X$. 
For any $\alpha>0$ the operator $(-A)^{-\alpha}$ belongs to $L(X)$, is injective and satisfies
\begin{eqnarray}
(-A)^{-\alpha}(-A)^{-\beta}=(-A)^{-(\alpha+\beta)},\quad \alpha,\beta>0.
\label{eq:Product-formula-fractional-powers}
\end{eqnarray}
\end{proposition}

%

\begin{definition}
\label{def:Positive-Fractional-powers-operator}
Let $A$ be the generator of an exponentially stable analytic semigroup over a Banach space $X$.
For $\alpha>0$ define $(-A)^{\alpha}$ as the inverse of $(-A)^{-\alpha}$ with 
$D((-A)^{\alpha}) = \im((-A)^{-\alpha})$.

By definition, we set $A^0:=I$.
\end{definition}

Let us collect several basic properties of the fractional powers; see \cite[Theorem 6.8, p. 72]{Paz83} for the first three items, and \cite[Exercises, p. 26]{Hen81} for the last one.
\begin{proposition}
\label{prop:Properties-of-positive-fractional-powers} 
Let $A$ be the generator of an exponentially stable analytic semigroup over a Banach space $X$.

\begin{enumerate}[label=(\roman*)]
	\item For $\alpha>0$ the operators $(-A)^\alpha$ are closed and densely defined.
	\item Whenever $\alpha > \beta>0$, we have $D((-A)^\alpha) \subset D((-A)^\beta)$.
	\item For all $\alpha,\beta\in\R$ it holds that 
	\begin{eqnarray}
	(-A)^{\alpha+\beta}x = (-A)^{\alpha}(-A)^{\beta}x,
	\label{eq:produce-rule-fractional}
	\end{eqnarray}
	for every $x \in D((-A)^\gamma)$, where $\gamma=\max\{\alpha,\beta,\alpha+\beta\}$.
	\item $(-A)^\alpha T(t) = T(t)(-A)^\alpha$ on $D((-A)^\alpha)$,\quad  $t\ge 0$.
\end{enumerate}
\end{proposition}
%
%
%
%
%

 Using fractional powers of operators, we can define the spaces that will be important to deal with
nonlinear equations governed by analytic semigroups.
\begin{definition}
\label{def:X-power-alpha}
Let $A$ be the generator of an analytic semigroup over a Banach space $X$.
Let $\alpha\ge 0$. Pick any $\omega>\omega_0(T)$ and define the space $X_\alpha$ and its norm as 
\begin{eqnarray}
X_\alpha:=D((\omega I-A)^\alpha),\quad \|x\|_{X_\alpha}:=\|(\omega I-A)^\alpha x\|_X,\quad x \in X_\alpha.
\label{eq:X-power-alpha}
\end{eqnarray}
\end{definition}
In particular, by definition $X_0 = X$, and $X_1 = D(A)$ (with norms as  in \eqref{eq:X-power-alpha}).

\begin{remark}
\label{rem:well-definiteness-X-power-alpha} 
The choice of different $\omega > \omega_0(T)$ induces the same linear space $X_\alpha$ endowed with an equivalent norm.
\end{remark}

The spaces $X_\alpha$ have a good structure:
\begin{proposition}
\label{prop:X-power-alpha-are-Banach-spaces} 
Let $A$ be the generator of an exponentially stable analytic semigroup over a Banach space $X$.
For any $\alpha\ge 0$, the space $X_\alpha$ is a Banach space. 
Furthermore, for all $\alpha\ge\beta\ge0$ the space $X_\alpha$ is a dense subspace of $X_\beta$, with continuous embedding.
\end{proposition}

Similarly to the space $X_{-1}$, we introduce
\begin{definition}
\label{def:X_-alpha}
For $\alpha>0$ and $A$ as above, define the spaces $X_{-\alpha}$ as the completion of $X$ with respect to the norm $x \mapsto \|(\omega I-A)^{-\alpha}x\|_X$.
Then we have for any $0<\beta<\alpha<1$ the following chain of continuous embeddings:
\[
X_{-1} \supset X_{-\alpha} \supset X_{-\beta} \supset X = X_0 \supset X_{\beta} \supset X_{\alpha} \supset X_1.
\]
\end{definition}

The following property holds:
\begin{proposition}
\label{prop:Analytic-semigroup-Bound-on-A^alphaT(t)} 
Let $T$ be an analytic semigroup on a Banach space $X$ with a growth bound $\omega_0(T)$ and the generator $A$.
Then for each $\omega,\kappa>\omega_0(T)$ and each $\alpha \in[0,1)$ we have $\im(T(t)) \subset X_\alpha$, and there is $C_\alpha>0$ such that 
\begin{eqnarray}
\|(\omega I-A)^\alpha T(t)\|  \leq \frac{C_\alpha}{t^\alpha}e^{\kappa t},\quad t>0.
\label{eq:Bound-on-A^alphaT(t)}
\end{eqnarray}
Furthermore, the map $t\mapsto (\omega I-A)^\alpha T(t)$ is continuous on $(0,+\infty)$ in the uniform operator topology.
\end{proposition}

%
%
%
%

\end{document}